\documentclass[11pt]{amsart}
\usepackage[all]{xy}
\usepackage{amssymb}
\usepackage{amsmath}
\usepackage{graphicx}
\usepackage[english]{babel}
\usepackage{epsfig}
\usepackage{enumerate}
\usepackage{color}

\usepackage{subcaption}

\swapnumbers

\usepackage[usenames,dvipsnames]{pstricks}

\theoremstyle{definition}
\newtheorem{definition}[subsubsection]{Definition}
\newtheorem{example}[subsubsection]{Example}
\newtheorem{notation}[subsubsection]{Notation}
\newtheorem{remark}[subsubsection]{Remark}
\theoremstyle{plain}
\newtheorem{theorem}[subsubsection]{Theorem}
\newtheorem{proposition}[subsubsection]{Proposition}
\newtheorem{lemma}[subsubsection]{Lemma}
\newtheorem{corollary}[subsubsection]{Corollary}

\def\K{{\mathbb{K}}}

\def\K{\mathbb{K}}
\def\G{\Gamma}

\begin{document}

\author{Stefan Forcey} \address[Stefan Forcey]{
    Department of Mathematics\\
    The University of Akron\\
    Akron, OH 44325-4002}
    \email{sf34@uakron.edu}  \urladdr{http://www.math.uakron.edu/\~{}sf34/}
\author[Mar\'\i a Ronco]{Mar\'ia Ronco}
\address{MOR: Instituto de Matem\'aticas y F\'\i sica,
Universidad de Talca, Avda. Lircay s/n, Talca, Chile} \email{mariaronco@inst-mat.utalca.cl}

\title{Algebraic structures on graph associahedra}
\thanks{M. Ronco\rq s work is supported by the Project FONDECYT Regular 1171902.}


\begin{abstract} M. Carr and S. Devadoss introduced in \cite{CD} the notion of {\it tubing} on a finite simple graph $\G$, in the context of configuration spaces on the Hilbert plane. To any finite simple graph $\G$ they associated a finite partially ordered set, whose elements are the tubings of $\G$ and whose geometric realization is a convex polytope ${\mathcal K}\G$, the graph-associahedron. For the complete graphs they recovered permutahedra, for linear graphs they got Stasheff\rq s associahedra, while for simple graph they obtained the standard simplexes.

The goal of the present work is to give an \emph{algebraic} description of graph associahedra. We introduce a substitution operation on tubings, which allows us to describe the set of faces of graph-associahedra as a free object, spanned by the set of all connected simple graphs, under operations given via connected subgraphs. The boundary maps of graph-associahedra defines natural derivations in this context.

Along the way, we introduce a topological interpretation of the graph tubings and our new operations. In the last section, we show that substitution of tubings may be understood in the context of M. Batanin and M. Markl's operadic categories.
\end{abstract}
\maketitle

\section*{Introduction} \label{section:introduction}

 Non-symmetric operads are algebraic structures defined by an infinite number of operations. They may be described, using natural substitution of plane rooted trees, as algebras over a monad in a category of graded vector spaces (see for instance \cite{GiKa}, \cite{MSS} and \cite{LV}). This notion was generalized by V. Dotsenko and A. Koroshkin in \cite{DK}, who defined another monad in the category of graded vector spaces; the algebras over this monad are called shuffle operads and play an important rule in the description of Gr\"obner basis for operads. In \cite{LR}, the notion of permutad was introduced, it describes a subclass of shuffle operads, using substitution on surjective maps between finite sets, instead of plane rooted trees.

The notion of operad has been generalized recently by M. Batanin and M. Markl in \cite{BaMa}, where they introduced operadic categories, and define operads in this context. Their construction covers a large range of algebraic structures which are not operads in the strict sense, but which are obtained by composing objects of certain type. As shown in \cite{BaMa1}, it is possible to pass from operadic categories Feynman categories, defined  by R. Kaufman and B. Ward in \cite{KaWa}.

The object of our work is graph associahedra. In \cite{CD} (and see also \cite{Deva}), M. Carr and S. Devadoss introduced the notion of {\it tubing} on a finite simple graph $\G$, and defined, for a fixed graph $\G$ with $n$ nodes, an order on the set ${\mbox{Tub}(\G)}$, of tubings of $\G$. They proved that the geometric realization of ${\mbox{Tub}(\G)}$ is always a convex polytope, whose dimension is ${\mbox{}n-1}$. \begin{enumerate}
\item When $\G$ is the path, or linear graph $L_n$, tubings in $L_n$ are in a one-to-one correspondence with planar rooted trees, and their geometric realization gives the Stasheff associahedra.
\item When $\G$ is the complete graph $K_n$, the tubings on $K_n$ describe the surjective maps between finite sets. In this case the polytopes ${\mathcal K}K_n$ are the permutohedra.
\item When $\G$ is the empty graph $C_n$, with no edges, the polytope ${\mathcal K}C_n$ is the standard simplex of dimension ${\mbox {}n-1}$.
\item When $\G$ is the cyclic graph ${\mathcal C}y_n$, the M. Carr and S. Devadoss polytope is the cyclohedron.\end{enumerate}

The vector spaces spanned by the faces of certain families of graph associahedra have interesting algebraic structures. The space spanned by the faces of Stasheff associahedra, as well as the one spanned by the faces of permutohedra, have natural structures of graded conilpotent Hopf algebras, which have been widely studied (see for instance \cite{MaRe}, \cite{LodRon}, \cite{Cha}, \cite{AgSo1}, \cite{AgSo2}, \cite{HiNoTh}, \cite{NoTh}). It is well known that the space spanned by the faces of cyclohedra has a natural structure of a Hopf module over the Hopf algebra spanned by the faces of the
Stasheff associahedra (see \cite{Dev}, \cite{FoSp}, \cite{HoLa}); and the space spanned by the faces of standard simplices describe free objects for the operads of trialgebras, as defined in \cite{Cha} and \cite{LodRon}.

All these results concern graph associahedra of certain families of graphs, which satisfy very particular conditions. For instance, subgraphs of complete graphs, as well as their reconnected complements, are complete graphs. The same property holds for linear graphs or simple graphs, but it is not a frequent feature in graphs.

However, it is possible to show that the space spanned by all the faces of graph associahedra has a natural structure of pre-Lie coalgebra.

When restricted to plane trees, the pre-Lie coalgebra structure is closely related to binary products, called {\it graftings}, as described by M. Gerstenhaber in \cite{Ger}. Given two plane trees $t$ and $w$, it is possible to graft the root of $w$ to any leaf of $t$. This binary operation is called the grafting of $w$ on the $i^{\text{th}}$ leaf of $t$, and is denote by $t\circ _i w$. The faces of Stasheff associahedra may be described in terms of plane trees, and the boundary map on those faces is a derivation for the products $\circ_i$s. From an algebraic point of view, the vector space
spanned by the faces of Stasheff associahedra is the free non-symmetric operad generated by one element in degree $n$, for $n\geq 1$.

Looking at trees as tubings on the linear graph, graftings arise naturally by inserting tubings into certain tubes. This procedure may be easily transferred to the faces of permutohedra. In \cite{LR}, a new type of colored operad or operad-like structure was introduced, defined in terms of graftings of surjective maps or tubings of the complete graphs. This structure, called shuffle algebra or permutad, gives a particular type of Dotsenko and Koroshkin\rq s shuffle operad. The graftings $\circ _i$ of a non-symetric operad are replaced in this case by binary products $\circ _{\sigma}$, where
$\sigma$ is a shuffle, and the boundary map of the permutohedron is described as a derivation for these new binary operations.

Our goal is to show that it is possible to give a general notion of substitution in Carr and Devadoss\rq s graph associahedra, in such a way that when we restrict ourselves to linear graphs, we obtain the monad defining non-symmetric operads; while when we restrict substitution to complete graphs, we get the monad defining permutads. These substitutions completely describe graph associahedra on connected simple finite graphs. The vector space ${\mbox {\bf Tub}}$, spanned by the faces of all graph associahedra on connected graphs, is a free algebra on the vector space spanned by all finite
connected simple graphs with a total order on its nodes, under binary operations $\circ_{\G, t}$, where $\G$ is a simple finite connected graph and $t$ is a connected subgraph of $\G$, under certain relations. Moreover, the boundary map of the graph associahedra is a derivation for the products $\circ_{\G, t}$.

For certain families of graphs, tubings may be described as leveled trees of different types, as described in \cite{For}, \cite{FoLaSo}, \cite{FoLaSo1} and \cite{BeFoRoSh}. In these cases, substitution may be recovered as the grafting of trees.

In \cite{Cha}, F. Chapoton described a non-symmetric differential operad ${\mbox {Trias}\rq}$, and showed that the standard simplices may be completely described as the free ${\mbox {Trias}\rq}$ algebra over the vector space of dimension one.  We show that the faces of graph associahedra associated to (non-connected) finite simple graphs are the free ${\mbox {Trias}\rq}$ algebra over ${\mbox {\bf Tub}}$.

Finally, we describe a strict operadic category ${\mathcal O}_{CD}$ in the sense of M. Batanin and M. Markl, whose objects are ordered finite sets of tubings, substitution on graph associahedra gives the operadic composition in this framework. The
full subcategory of ${\mathcal O}_{CD}$ whose objects are tubings on complete graphs gives the category ${\mbox{Per}}$, introduced in \cite{BaMa1}.

In the first section we recall the main definitions of graph associahedra and set the notations that we need in the sequel. We introduce a new interpretation of tubings as finite topologies which respect the graph structure. The last subsection is devoted to the description of the pre-Lie coproduct defined on the space of faces of graph associahedra.

Section 2 contains the definition of substitution on tubings, and its main properties, such as associativity.

In the third section, we give a complete description of graph associahedra in terms of generators and relations. We prove that the vector space ${\mbox {\bf Tub}}$, spanned by tubings on connected simple finite graphs, is spanned by the set of all connected simple finite graphs under the operations $(\G, t)$, for $t\in {\mbox {Tub}(\G)}$. When we consider the restriction of substitution to linear graphs, we recover non-symmetric operads (or M. Gerstenhaber pre-Lie systems), while when we consider the substitution on the family of all complete graphs, we get permutads, as defined in \cite{LodRon}.

In section 4 we prove that the vector space ${\mbox {\bf DTub}}$, is equipped with the structure of a ${\mbox {Trias}\rq}$ algebra and that it is the free ${\mbox {Trias}\rq}$ algebra spanned by the vector space ${\mbox {\bf Tub}}$. We also show that the binary operations defined in ${\mbox {\bf DTub}}$, induce an $L$-algebra structure on ${\mbox {\bf Tub}}$, for the $L$ algebras defined by P. Leroux in \cite{Ler}.

Finally, in section 5, we recall M. Batanin and M. Markl\rq s definition of operadic category, and construct the operadic category ${\mathcal O}_{CD}$ whose objects are finite families of tubings. We prove that the substitution of tubings defines an operad in the category ${\mathcal O}_{CD}$.

All our work deals with finite graphs, equipped with a total order on its set of nodes. However, the Carr and Devadoss polytope ${\mathcal K}\G$ does not depend on the way of numbering the nodes of a graph $\G$. In fact, we need the total order on the set of nodes of a graph $\G$ uniquely to define the boundary map of ${\mathcal K}\G$ in terms of reconnected complements and substitution on graphs, and in the definition of the operadic category ${\mathcal O}_{CD}$.
\bigskip

\section{Preliminaries on graph associahedra}
\subsection{Graphs}
\medskip

All the vector spaces considered in the present work are over a field $\K$. For any set $X$, we denote by $\K[X]$ the vector space spanned by $X$. For any positive integer $n$, we denote by $[n]$ the set $\{1,\dots ,n\}$. The graphs $\G$ we deal with in the manuscript, excepted those in section 4, are simple finite and connected, and their set of nodes is totally ordered.
\medskip

Let $\G$ be a graph, whose set of nodes ${\mbox {Nod}(\G)}$ is $[n]$, we denote by ${\mbox {Edg}(\G)}$ the set of its edges. Any edge $e\in {\mbox {Edg}(\G)}$ is identified with its (unordered) pair of extremes $e=\{v,w\}$.

\begin{definition} \label{connectedcomplement} A \emph{tube} is a set of nodes of $\G$ whose induced graph is a connected subgraph of $\G$.
For a given tube $t$ and a graph $\G$, let $\G_t$ denote the induced subgraph on the graph $\G$. When $\G$ is connected, the set of all nodes of $\G$ determines the {\it universal tube} $t_{\G}$.

For any tube $t\in \G$, the {\it reconnected complement} $\G_t^*$ is the graph whose set of nodes is ${\mbox {Nod}(\G)}\setminus \{t\}$, and whose edges are determined by the pairs of nodes $\{v,w \}$ satisfying one of the following conditions:\begin{enumerate}[(a)]
\item the edge $\{ v,w\}$ belongs to ${\mbox {Edg}(\G)}$,
\item there exists nodes $u$ and $u\rq$ in $t$ such that $\{v,u\}$ and $\{ w ,u\rq\}$ belong to ${\mbox {Edg}(\G)}$.\end{enumerate}\end{definition}
\medskip

We often renumber the nodes of $\G_t^*$ and of $\G_t$ starting at 1 but preserving the order between nodes. We use the same terminology and notation for the renumbered version, when there is no chance of confusion.  In the (renumbered) $\G_t^*$, for any pair of nodes $v,w$ in ${\mbox {Nod}(\G)-t}$, we have $v<w$ in $\G$ if, and only if, $v<w$ in $\G_t^*$. The same condition holds  for the renumbered $G_t,$ for nodes in $t.$
\medskip

\begin{example} Let $\G$ be the graph

\begin{center}
\includegraphics{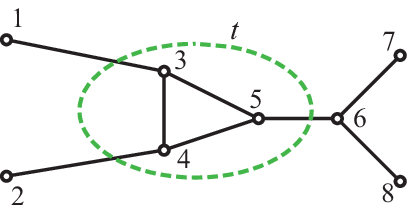}
\end{center}

where $t=\{ 3, 4,5\}$. Here, from left to right, are $\G_t$, its renumbered version, $\G_t^*$ and its renumbered version:
\medskip
\begin{center}
\includegraphics{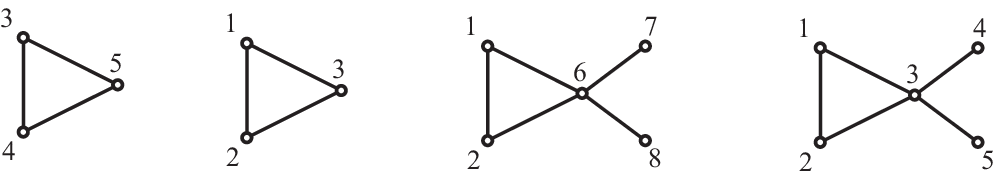}
\end{center}

\end{example}
\medskip

When $\G =\G_1\coprod \dots \coprod\G_k$ is the disjoint union of connected graphs $\G_i$, for $1\leq i\leq k$, and $t$ is a connected subgraph of $\G_{i_0}$, the reconnected complement $\G_t^*$ is the graph $\G_1\coprod \dots \coprod (\G_{i_0})_t^*\coprod \dots \coprod \G_k$.

\bigskip

\subsection{Basic constructions on Carr and Devadoss\rq s graph associahedra}
\medskip

We give a brief description of graph associahedra as introduced by M. Carr and S. Devadoss in \cite{CD}; for a complete description of their construction and further details we refer to their work.
\medskip

\begin{definition}\label{compattubes} Let $t$ and $t\rq $ be two different tubes in a simple finite graph $\G$. We say that:\begin{enumerate}
\item $t$ and $t\rq$ are \emph{nested} if  $t\subseteq t\rq$.
\item $t$ and $t\rq$ are \emph{far apart} if $t \cup t\rq$ is not a tube in $\G,$ that is, the induced subgraph of the union
 is not connected, (equivalently none of the nodes of $t$ are adjacent to a node of $t\rq$).
 \item  $t$ and $t\rq $ are {\it compatible} if they are either nested or far apart.
 \item $t$ and $t\rq$ are {\it linked} in $\G$ if they are disjoint and incompatible (that is, disjoint, and there exists an edge $e\in {\mbox {Edg}(\G)}$ which has one extreme in $t$ and the other one in  $t\rq $).
\end{enumerate}
\end{definition}
\medskip

\begin{definition} \label{defntubing} Let $\G$ be a simple connected finite graph. A {\it tubing} in $\G$ is a non-empty family $T=\{ t^i\} _{1\leq i\leq k}$ of tubes, which contains the universal tube $t_{\G}$, and such that every pair of tubes in $T$ is compatible.\end{definition}
\medskip

Note that a tubing of a connected graph covers the set of nodes of that graph by virtue of containing the universal tube, but no collection of proper tubes can cover the set of nodes. Note that a tubing $T$ in a connected graph $\G$ contains at most $n$ tubes if $\G$ has $n$ nodes. A tubing with $k$ tubes is called a {\it $k$-tubing} on $\G$, for $1\leq k\leq n$. We have many examples pictured throughout, in which tubes are shown as circled subgraphs. For simplicity the universal tube is not shown in the pictures of tubings, just assumed to be included. The only examples pictured of
non-tubings are in Remark~\ref{nonex}.
\medskip

\begin{notation} \label{notnmaxtubings} Let ${\mbox {Tub}(\G)}$ denote the set of all tubings of a connected graph $\G$.
\begin{enumerate}\item For any tubing  $T\in {\mbox {Tub}(\G)}$, we denote by ${\overline T}:=T\setminus \{t_{\G}\}$ the set of proper tubes of $T$.
\item A {\it maximal tube} of a tubing $T\in {\mbox {Tub}(\G)}$ is a tube $t\in {\overline T}$ which is not
contained in any other tube of ${\overline T}$. We denote by ${\mbox {Maxt}(T)}$ the set of maximal tubes of $T$.
\item In order to simplify notation, we denote by $T_{\G}$ the tubing whose unique tube is $t_{\G}$.
\item We say that two tubings $T$ and $T\rq$ of $\G$ are {\it compatible} if the union of $T$ and $T\rq$ is also a tubing of $\G$.
\end{enumerate}
\end{notation}
\medskip

\begin{remark}\label{rem:restinduc} Let $T$ be a tubing in a graph $\G$. \begin{enumerate}
\item For any tube $t$ in $\G$, the restriction of $T$ to $\G_t$ is either a tubing of $\G_t$ or the empty tube, we denote it by $T\vert_t$. When $t\in T$, the tubes of $T\vert_t$ are all the tubes $t\rq \in T$ such that
$t\rq \subseteq t$.
\item For any tube $t$ in $\G$ such that $\{T,t\}$ is a tubing in $\G$, $T$ induces a tubing $T_t^*$ in the reconnected complement $\G_t^*$. A tube $t\rq$ of $T_t^*$ is either a tube $t\rq $ of $T$ such that $t\rq \cap t = \emptyset$, or $t\rq$ is the restriction of a tube which contains $t$ to the set of nodes which do not belong to $t$.

\end{enumerate}\end{remark}
\medskip
\begin{remark}\label{nonex}
Note that the requirement for $\{T,t\}$ to be a tubing in $\G$ is necessary for the definition of the tubing $T_t^*$. Consider for instance the graph
\medskip

\begin{figure}[h]
\includegraphics{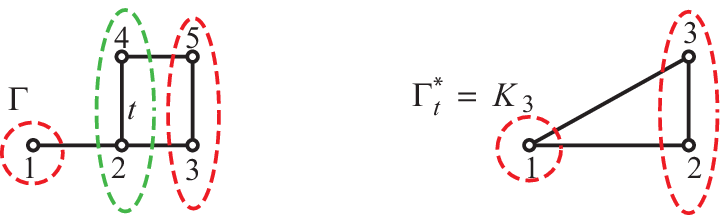}
\end{figure}

with the tube $t=\{ 2,4\}$. The tubing $T=\{ \{1\}, \{3,5\}, \{1,2,3,4,5\}\}$ does not induce a tubing on $\G_t^*$.
\end{remark}

\subsubsection{Topological interpretation}\label{stefan}
We show here that the tubings of connected graphs are precisely topological bases which 1) are comprised of tubes, and 2) generate a topology on nodes which respects the connectivity of the graph, and its reconnected complements. Recall that a space is \emph{topologically connected} if there are no two disjoint open sets which cover the space. Recall that a \emph{subspace topology} is formed by all intersections of any open set with the subspace. Recall that a \emph{topological basis}  on a set $X$ is a collection of subsets of $X$, which covers $X$, and such that for any two basis elements
their intersection $I$ is itself covered by basis elements that are subsets of $I.$  A basis on $X$ generates a unique topology on $X$.

Let $T$ be a collection of tubes on a connected graph $\Gamma.$
 \begin{theorem}\label{topo}
 $T$ is a tubing on $\Gamma$ if, and only if, $T$ is a topological basis on ${\mbox { \emph{Nod}}(\G)}$ such that the
following \emph{connectivity condition} holds: for all pairs $\{v,v'\}$ of nodes, if $\{v,v'\}$ is an edge in $\Gamma$ or in $\Gamma^*_t$ for
 $ t\in T$,  then $\{v,v'\}$ is connected as a subspace in the topology generated by $T.$
\end{theorem}
\begin{proof}
First we show that if $T$ is a basis obeying our connectivity condition, then $T$ is a tubing. We use strong induction on the number of nodes in $\Gamma$. The case for one node is trivial. We assume the implication holds for $k < n$ nodes, and prove the case of $\Gamma$ with $n$ nodes. To show that the connectivity condition implies compatibility of tubes, we show the contrapositive:  if tubes $s,t\in T$ are incompatible, then there is an edge $\{v,v'\}$ (either in $\Gamma$ or a reconnected complement) that is disconnected as a subspace. There are two cases when $s,t\in T$ are incompatible:
\begin{enumerate}
\item[($i$)]~~$s$ and $t$ are linked. Taking the
nodes that are connected by the linking edge to be $v\in s$ and $v'\in t,$ we see that $\{v,v'\}$ has the sets $\{v\}$ and $\{v'\}$ in its subspace topology, so it is disconnected.
\item[($ii$)]~~$s$ and $t$ are  intersecting (not nested). Then since $T$ is a basis, $ u= s\cap t$
must be covered by subsets of $u$ that are tubes in $T.$  In fact, $u$ is connected and the tubes in $T\vert_u$ form a topological basis for $u$, which inherits the connectivity condition. Thus by the inductive assumption, these tubes in $u$ will be compatible. Therefore $u$ itself must be a tube in $T$ in order to achieve a cover.  Then letting $v\in s-t$ be connected by an edge to $u$ and $v'\in t-s$ be connected by an edge to $u$, we have that $\{v,v'\}$ is an edge in $\Gamma^*_u$ and simultaneously that the subspace topology of $\{v,v'\}$ is disconnected.
\end{enumerate}
Therefore our basis elements of $T$ are tubes which are compatible. Since $T$ is a basis, the tubes cover ${\mbox {Nod}(\G)},$ but since they are compatible this means the universal tube must be included as an element of $T$, so our basis is a tubing.

Finally we show the other implication: if $T$ is a tubing then $T$ is a topological basis that meets our connectivity condition. First, $T$ covers ${\mbox {Nod}(\G)}$ since the universal tube is in $T$ by definition.  Second, pairwise intersections of tubes are covered by the tubes inside them---since the only intersections of tubes are empty or (nested) tubes in $T$. We show that this topological basis obeys the connectivity condition. Again we prove the contrapositive: we show that if the connectivity condition does not hold then there is a pair of incompatible tubes. Again there are two
cases:
\begin{enumerate}
\item[($i$)]~~let
$\{v,v'\}$ be an edge in $\Gamma$ and assume that the subspace of those two nodes is disconnected. Therefore there must be a pair of disjoint tubes in $T$, one containing $v$ and the other $v'$, which are thus linked, so incompatible.
\item[($ii$)]~~assume $\{v,v'\}$ is not an edge in $\Gamma$, but
$\{v,v'\}$ is an edge of $\Gamma^*_t$ for some tube $t.$ Therefore both $v$ and $v'$ are connected by an edge to $t$. Also assume that the subspace of those two nodes is disconnected. Then there must be a pair of tubes in $T$  one containing $v$ and the other $v'$ (but neither containing both nodes). If either tube in the pair is incompatible with $t,$ we are done. If not, then both tubes must contain $t$, and thus are intersecting but not nested, so incompatible with each other.
\end{enumerate}
\end{proof}

For a finite topological space, being connected implies being path-connected. Therefore we have the easy corollary to Theorem~\ref{topo}:
\begin{corollary}\label{paths}
If $T$ is a tubing on $\G,$ then any pair of nodes connected by a path (a sequence of edges in the graph $\G$) are also connected by a topological path (continuous map from [0,1] to ${\mbox {Nod}(\G)}$)  in the topology generated by $T.$ Moreover, given the graphical path, there exists a topological path whose range is precisely the nodes in the graphical path.
\end{corollary}

When a graph is not connected, but rather contains multiple connected components, there are two existing definitions of tubing. In this paper, we follow \cite{CD} by requiring some of the connected components to not be themselves elements of the tubing:

\begin{definition} \label{tubingdisconnected} Let $\G =\G_1\coprod \dots \coprod \G_k$ be a non-connected simple graph, where $\G_i$ is a connected component, for $1\leq i\leq k$ and $2\leq k$. A tubing $T$ of $\G$ is a list  $(T_1,\dots ,T_k) $, satisfying that:\begin{enumerate}
\item  there exists a non-empty set $i_1<\dots <i_s$ such that $T_{i_j} ={\overline {W}}_{i_j}$, with $W_{i_j} \in {\mbox {Tub}(\G_{i_j})}$,
\item for $i\notin \{i_1, \dots, i_s\}$, $T_i\in  {\mbox {Tub}(\G_i)}$.\end{enumerate} .\end{definition}

\begin{remark} \label{rem:restriction} Let $\G$ be a graph and let $\Omega $ be another graph with the same set of nodes than $\G$ and such that ${\mbox {Edg}(\Omega)}\subseteq  {\mbox {Edg}(\G)}$, there exists a natural surjective map ${\mbox {res}_{\Omega}^{\G}}:{\mbox {Tub}(\G)}\longrightarrow {\mbox {Tub}(\Omega)}$. Given three graphs $\Theta$, $\Omega$ and $\G$ with the same number of nodes and such that
${\mbox {Edg}(\Theta)}\subseteq {\mbox {Edg}(\Omega)}\subseteq  {\mbox {Edg}(\G)}$, it is immediate to verify that the composition ${\mbox {res}_{\Theta}^{\Omega}}\circ {\mbox {res}_{\Omega}^{\G}}$ coincides with the application ${\mbox {res}_{\Theta}^{\G}}$.\end{remark}

For instance, if $\G = K_3$ and $\Omega = L_3$ is the linear graph as follows:
\begin{center}
   \includegraphics{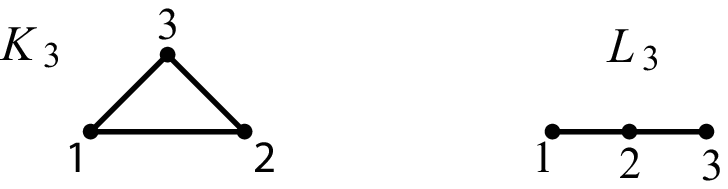}
\end{center}
then the tubings $T= \{ \{1\}, \{ 1,3\}, \{1,2,3\}\}$ and $T\rq = \{ \{3\}, \{ 1,3\}, \{1,2,3\}\}$ on $K_3$  both map to the same tubing: $\{ \{1\}, \{3\}, \{1,2,3\}\}$ on $L_3$.

\begin{definition} \label{ordertubings} For any graph $\G$, the set ${\mbox {Tub}(\G)}$ of tubings of $\G$ is partially ordered by the relation:
\begin{equation*}T\preceq T\rq\ {\rm if}\ T\ {\rm is\ obtained\ from\ } T\rq \ {\rm by\ adding\ compatible\ tubes}.\end{equation*}\end{definition}
\medskip

In \cite{CD}, M. Carr and S. Devadoss proved that, for any simple finite graph $\G$, the geometric realization of the partially ordered set $({\mbox {Tub}(\G)}, \preceq)$ is the barycentric division of a simple, convex polytope ${\mathcal K}\G$ of dimension $n$-$1$, where $n$ is the number of nodes of $\G$, whose faces of dimension $r$ are indexed by the $n$-$r$ tubings of $\G$, for $1\leq r\leq n$.

The vertices of the polytope ${\mathcal K}\G$ coincide with the subset ${\mbox {MTub}(\G)}$ of minimal tubings, while the universal tubing $T_{\G}$ corresponds to the $(n{\it -}1)$-cell of ${\mathcal K}\G$.
\medskip

\begin{example} \label{excomplete graph} For the complete graph $K_n$, with set of nodes $[n]$ and edges ${\mbox {Edg}(K_n)}=\{ (i,j)\mid 1\leq i<j\leq n\}$, any pair of nodes is an edge. Therefore, any tubing $T$ of $K_n$ may be identified with a sequence of subsets $t^1\subsetneq t^2\subsetneq \dots \subsetneq t^r\subsetneq [n]$.

As shown in \cite{Deva}, there exists a natural bijection between tubings of $K_n$ and surjective maps $x: [n]\longrightarrow [r]$, for $1\leq r\leq n$, given by:
\begin{equation*}x_T(i)=j, {\rm for}\ i\in t^j\setminus t^{j-1},\end{equation*}
where $T= \{t^1\subsetneq t^2\subsetneq \dots \subsetneq t^r\}$.

In this case, the polytope ${\mathcal K}K_n$ is the permutohedron of dimension $n{-}1$.\end{example}
\medskip

Any simple graph $\G$ with $n$ nodes, may be obtained from $K_n$ by eliminating some edges.  The surjective map ${\mbox {res} _{\G}^{K_n}}$ describes the way faces of the permutohedron are contracted in order to obtain the polytope ${\mathcal K}\G$.
\medskip

The following result is the original characterization of faces of the graph associahedra from Carr and Devadoss (see Theorem 2.9 of \cite{CD}).

\begin{theorem}\label{faces} For any graph $\G$ and any proper tube $t$ in it, the facet of ${\mathcal K}\G$ corresponding to the tubing $\{t\}$ is isomorphic to the product ${\mathcal K}\G _t \times {\mathcal K}\G _t^*$.\end{theorem}

We refer to \cite{CD}, \cite{Deva} and \cite{DeFo} for more descriptions of the polytopes ${\mathcal K}\G$.
 \bigskip

\subsection{The pre-Lie coproduct on tubings}
\medskip

 We denote by ${\mbox {\bf Tub}}$ the $\K$-vector space spanned by all the pairs $(\G, T)$, where $\G$ is a finite connected simple graph, equipped with a total order on the its set of nodes, and $T$ is a tubing of $\G$.

Let us recall the definition of right pre-Lie algebra (also called right symmetric algebra). For a complete description of the properties and examples of this type of algebras we refer to D. Burde\rq s work \cite{Bur}.

\begin{definition} \label{rightpreLie} A {\it right pre-Lie}, or {\it right symmetric}, algebra is a vector space $S$ equipped with a binary product $\bullet: S\otimes S\longrightarrow S$, satisfying that:
\begin{equation*}x\bullet (y\bullet z)\ {\mbox{-}}\ (x\bullet y)\bullet z = y\bullet (x\bullet z)\ {\mbox{-}}\ (y\bullet x)\bullet z,\end{equation*}
for any elements $x,y,z\in S$. \end{definition}

Clearly, any associative algebra is a right pre-Lie algebra. Moreover, the binary operation $[ x, y] := x\bullet y\ {\mbox{-}}\ y\bullet x$ defines a Lie bracket on $S$.
\medskip

Dualizing the notion of right pre-Lie algebra, we get that a right pre-Lie coalgebra is a vector space $C$ equipped with a coproduct $\Delta_{\bullet}: V\longrightarrow V\otimes V$ satisfying the relation:
\begin{equation*}(Id_C\otimes \Delta_{\bullet}\ {\mbox{-}}\ \Delta_{\bullet}\otimes Id_C)\circ \Delta_{\bullet} (x) = (\tau\otimes Id_C)\circ (Id_C\otimes \Delta_{\bullet}\ {\mbox{-}}\ \Delta_{\bullet}\otimes Id_C)\circ \Delta_{\bullet} (x),\end{equation*}
for any $x\in C$, where $\tau$ denotes the twisting $\tau(x\otimes y) := y\otimes x$ on $C\otimes C$.
\medskip

Theorem \ref{faces} induces the following definition.

\begin{definition} \label{coproduct} Let $T$ be a tubing of a connected graph $\G$, define $\Delta_{\bullet} (T)\in {\mbox {\bf Tub}^+}\otimes {\mbox {\bf Tub}^+}$ as follows:
\begin{equation*}\Delta_{\bullet}(T):=\sum _{t\in T}T\vert_t\otimes T_t^*,\end{equation*}
where the sum is taken over all the tubes $t$ in $T$ and on the empty tube, which is identified with the unit $1_{\K}$ of $\K$. We define completely the coproduct  $\Delta_{\bullet}$ on the vector space ${\mbox {\bf Tub}^+}$ by setting that $\Delta_{\bullet}(1_{\K}) = 1_{\K}\otimes 1_{\K}$, and extending it by linearity.\end{definition}

\begin{lemma}\label{preLiecoprod} The coproduct $\Delta_{\bullet}$ defines a right pre-Lie coalgebra structure on ${\mbox {\bf Tub}}$.\end{lemma}
\medskip

\begin{proof} It is easy to see that
\begin{equation*}(Id_C\otimes \Delta_{\bullet}\ {\mbox{-}}\ \Delta_{\bullet}\otimes Id_C)\circ \Delta_{\bullet} (T) =\sum _{{t_1,t_2\in T}\atop{t_1\cap t_2=\emptyset}} T\vert_{t_1}\otimes T\vert_{t_2}\otimes
T_{\{t_1,t_2\}}^*,\end{equation*} where the sum is taken over all pairs of non-empty tubes $(t_1,t_2)$ of $T$ satisfying that $t_1\cap t_2$ is empty.

Clearly, if the pair $(t_1,t_2)$ satisfies that $t_1\cap t_2 = \emptyset$, then the pair $(t_2,t_1)$ also satisfies the condition, which implies the result.\end{proof}
\bigskip

\section{Substitution on tubings}

The definition of operad involves the composition of operations, which may be described by grafting the root of a colored tree in a leaf of another one. In particular, non-symmetric operads (also called Gerstenhaber pre-Lie systems) are based on grafting of plane rooted trees. This notion was generalized to surjective maps in \cite{LR}, and resulted in descriptions of some particular families of shuffle operads, introduced by V. Dotsenko and A. Koroshkin in \cite{DK}, called permutads. Plane binary trees are in bijection with the tubings of linear graphs, and surjective maps with the tubings
of complete graphs, so in these cases we may describe grafting in terms of tubings. In fact, this notion may be  defined on the faces of graph associahedra in a very natural way; we call this operation {\it substitution}.
\medskip

\subsection{Basic definitions}
\medskip

We want to formalize the notion of composing tubes, into a tubing. In order to do that properly we need some previous results.
\medskip

\begin{lemma} \label{separatedtubes} Let $\G$ be a graph. For any pair of tubes $t$ and $t\rq$ in $\G$, we have that $(\G_t^*)_{t\rq}^* = (\G_{t\rq}^*)_t^*$.\end{lemma}

\begin{proof} Clearly the sets of nodes of the graphs $(\G_t^*)_{t\rq}^*$ and $(\G_{t\rq}^*)_t^*$ are both equal to ${\mbox {Nod}(\G)}\setminus (t\cup t\rq)$, where $\cup $ denotes the union.

On the other hand, for any pair of nodes $v$ and $w$ in ${\mbox {Nod}(\G)}\setminus (t\cup t\rq)$, we have that the pair $(v, w)$ is an edge of $(\G_t^*)_{t\rq}^*$ if, and only if, at least one of the following situations is fulfilled:\begin{enumerate}
\item $(v, w)$ is an edge in $\G$,
\item there exist edges $(v, v\rq)$ and $(w, w\rq)$ in $\G$ such that $v\rq$ and $w\rq$ belong to $t$,
\item there exist edges $(v, v\rq)$ and $(w, w\rq)$ in $\G$ such that $v\rq$ and $w\rq$ belong to $t\rq$,
\item there exist an edge $(u,z)$ in $\G$ such that $u\in t$ and $z\in t\rq$, and edges $(v, v\rq)$ and $(w, w\rq)$ in $\G$ satisfying that either $v\rq\in t$ and $w\rq\in t\rq$, or
$v\rq \in t\rq$ and $w\rq \in t$.\end{enumerate}

Note that, exchanging $t$ and $t\rq$, we get that the edges of $(\G_{t\rq}^*)_t^*$ are determined in the same way, so both graphs coincide.\end{proof}
\medskip

Lemma \ref{separatedtubes} allows us to introduce the following definition.

\begin{definition}\label{maxtubings} Let ${\mbox{Maxt}(T)}=\{t^1,\dots ,t^m\}$ be the set of maximal proper tubes of a tubing $T$ in $\G$. We denote by $\G_{T}^*$ the graph $(((\dots (\G_{t^1}^*)_{t^2}^*)\dots )_{t^m}^*$ obtained by iterating the construction of the reconnected complement. \end{definition}
\medskip

The proof of the following result is immediate.

\begin{lemma} \label{rem:restrictionreconnected} For any pair of graphs $\G$ and $\Omega$ with the same number of nodes, satisfying that ${\mbox {Edg}(\Omega)}\subseteq {\mbox {Edg}(\G)}$, using the same notations that in Remark \ref{rem:restriction}, we have that:\begin{enumerate}
\item For any tube $t$ in $\G$, the graphs $\G_t^*$ and $\Omega _{\mbox {res}_{\Omega}^{\G}(t)}^*$ have the same number of nodes, and ${\mbox {Edg}(\Omega _{\mbox {res}_{\Omega}^{\G}(t)}^*)}\subseteq {\mbox {Edg}(\G_t^*)}$.
\item For any tube $t$ in $\G$ and any tubing $S$ of $\G_t^*$, $S$ induces a tubing of $\Omega _{\mbox {res}_{\Omega}^{\G}(t)}^*$, denoted ${\mbox {res}_{\Omega}^{\G}}(S)_t^*$.
\end{enumerate}\end{lemma}
\medskip

\begin{definition}\label{defn:inductions} Let $\G$ be a graph and let $T$ and $T\rq$ be tubings of $\G$, and let $t\in T$ be a tube. \begin{enumerate}
\item Let $S$ be a tubing of $\G_t$ which is compatible with $T\vert _t$. Define $T\circ _tS$ as the set of tubes $s$ satisfying that $s\in T$, or $s\in S$ (considered as a tube in $\G$).
\item Let $S$ be a tubing of $\G_t^*$ which is compatible with $T_t^*$. A tube $s$ in $\G$ belongs to the set of tubes $T\diamond S$ if it satisfies one of the following conditions\begin{enumerate}
\item $s\in T$,
\item $s \in S$ and, as a tube in $\G$, $s$ is not linked to $t$,
\item $s = s\rq \cup t$, for some $s\rq \in S$ such that $s'$ is the union of some tubes all linked to $t$ in $\G$.\end{enumerate}
\end{enumerate} \end{definition}

\begin{lemma} \label{lem:inductubings} Let $T$ be a tubing of a graph $\G$, and let $t$ be a tube of $T$. We have that:\begin{enumerate}
\item if $S$ is a tubing of $\G_t$ compatible with $T\vert _t$, then $T\circ_t S$ is a tubing of $\G$,
\item if $S$ is a tubing of $\G_t^*$ compatible with $T_t^*$, then $T\diamond S$ is a tubing of $\G$.\end{enumerate}\end{lemma}

\begin{proof} The first point is evident.

For the second one, consider a pair of tubes $s$ and $w$ in $T\diamond S$,\begin{itemize}
\item if $s$ and $w$ are in $T$, then $s$ and $w$ are compatible,
\item if $s$ and $w$ do not belong to $T$ and are not linked to $t$, then both belong to $S$ and they are compatible.
\item if $s= s\rq \cup t \in S$, with $s\rq\in S$ linked to $t$, and $w\in T$, with  $w\cap t = \emptyset$, then $w$ and $s\rq$ are compatible in $\G_t^*$ and $w$ and $t$ are compatible in $\G$. So, $s$ and $w$ are compatible in $\G$.
\item if $s= s\rq \cup t \in S$, with $s\rq\in S$ linked to $t$, and $t\subset w\in T$, then $s\rq \cap (w\setminus t)\neq \emptyset$. As $s\rq$ and $w\setminus t$ are compatible in $\G_t^*$, we get that one of them must be contained
in the other one, and therefore either $s\subset w$ or $w\subset s$ and they are compatible.
\end{itemize}
So, any pair of tubes in $T\diamond S$ is compatible, and the result is proved.\end{proof}

Recall that in Definition \ref{maxtubings}, for any tubing $T$ of $\G$, we have defined $\G_T^*$  as the reconnected complement of $\G$ after taking off all the maximal proper tubes of $T$. The following definition extends the notion of induced tubings to this context.

\begin{definition}\label{defn:substitution} Let $T$ be a tubing of $\G$.\begin{enumerate}
\item Given a tubing $S$ of $\G_T^*$,  define the {\it substitution} of $S$ in $T$ to be the tubing $T\bullet _{{\G}} S$ of $\G$ whose elements are the tubes $t$ satisfying one of the following conditions:\begin{enumerate}
\item $t$ belongs to $T$,
\item $ t = s\coprod t^{i_1}\dots \coprod t^{i_k}$, where $s$ is a tube of $S$ and $\{ t^{i_1},\dots ,t^{i_k}\} $ is the set of maximal proper tubes of $T$ which are linked to at least one of a union of tubes which comprise $s$ in $\G$.\end{enumerate}
\item For any tube $t\in T$ and any tubing $S\in {\mbox {Tub}((\G_t)_{T\vert _{t}}^*)}$, the $t$-{\it substitution} of $S$ in $T$ is the tubing
$T\circ _t (T\vert _t \bullet _{\G_t}S)$ on $\G$. We denote it simply by $\gamma_t(T;S)$. \end{enumerate}\end{definition}

Note that, when $\G$ is connected and $t=t_{\G}$, we get $\gamma_{t_{\G}}(T;S)=T\bullet _{\G}S$. Moreover, if $T=\{ t_{\G}, t\}$ for a proper connected subgraph $t$ of $\G$, we have that $\gamma_t(T;S) = T\diamond S$, for any tubing $S\in \G_t^*$.

Lemma~\ref{separatedtubes} and a recursive argument on Lemma~\ref{lem:inductubings} show that the substitution is well-defined.
\medskip

Before going on, let us give some examples.

\begin{example} $(1)$ Let $\G$ be the linear graph $L_4$.

Consider the tubings $T=\{ \{2\}, \{4\}, \{1,2,3,4\}\}$ of $L_4$ and $S=\{\{1\}, \{1,2\}\}$ of  $(L_4)_T^*$
\medskip

\begin{figure}[h!]
\includegraphics{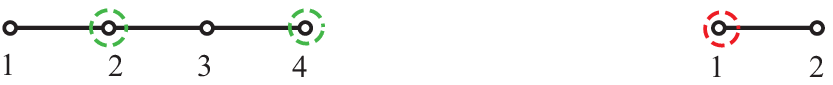}
\end{figure}
\medskip

The tubing $\gamma_{t_{L_4}}(T;S)$ is:
\medskip

\begin{figure}[h]
\includegraphics{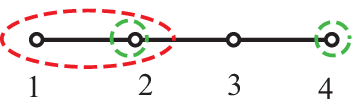}
\end{figure}
\medskip
\newpage
$(2)$ Here are some of the other operations on an arbitrary graph:
\medskip

\includegraphics[width=5.5in]{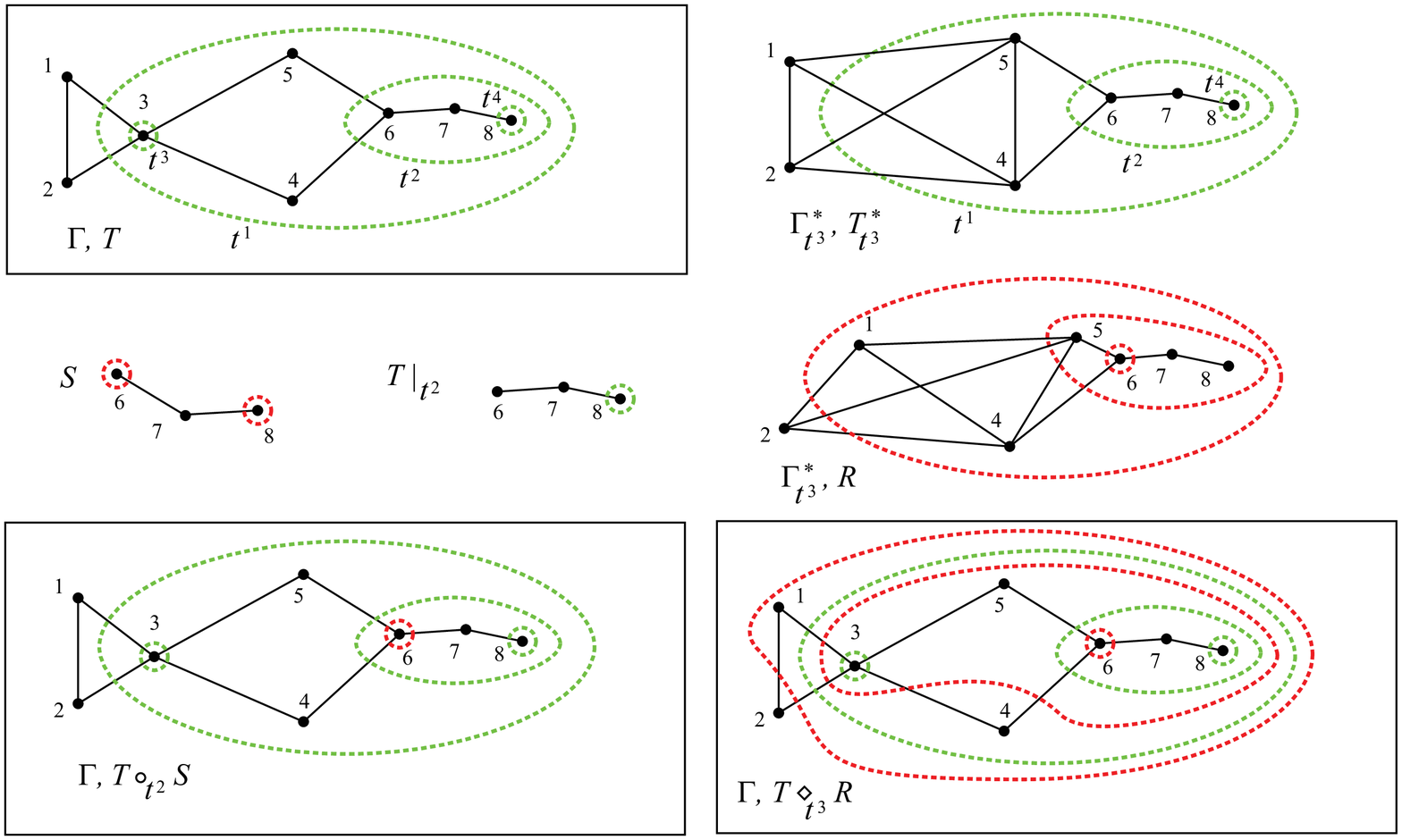}
\medskip

$(3)$ Here is the single substitution on that graph:
\medskip

\includegraphics[width=5in]{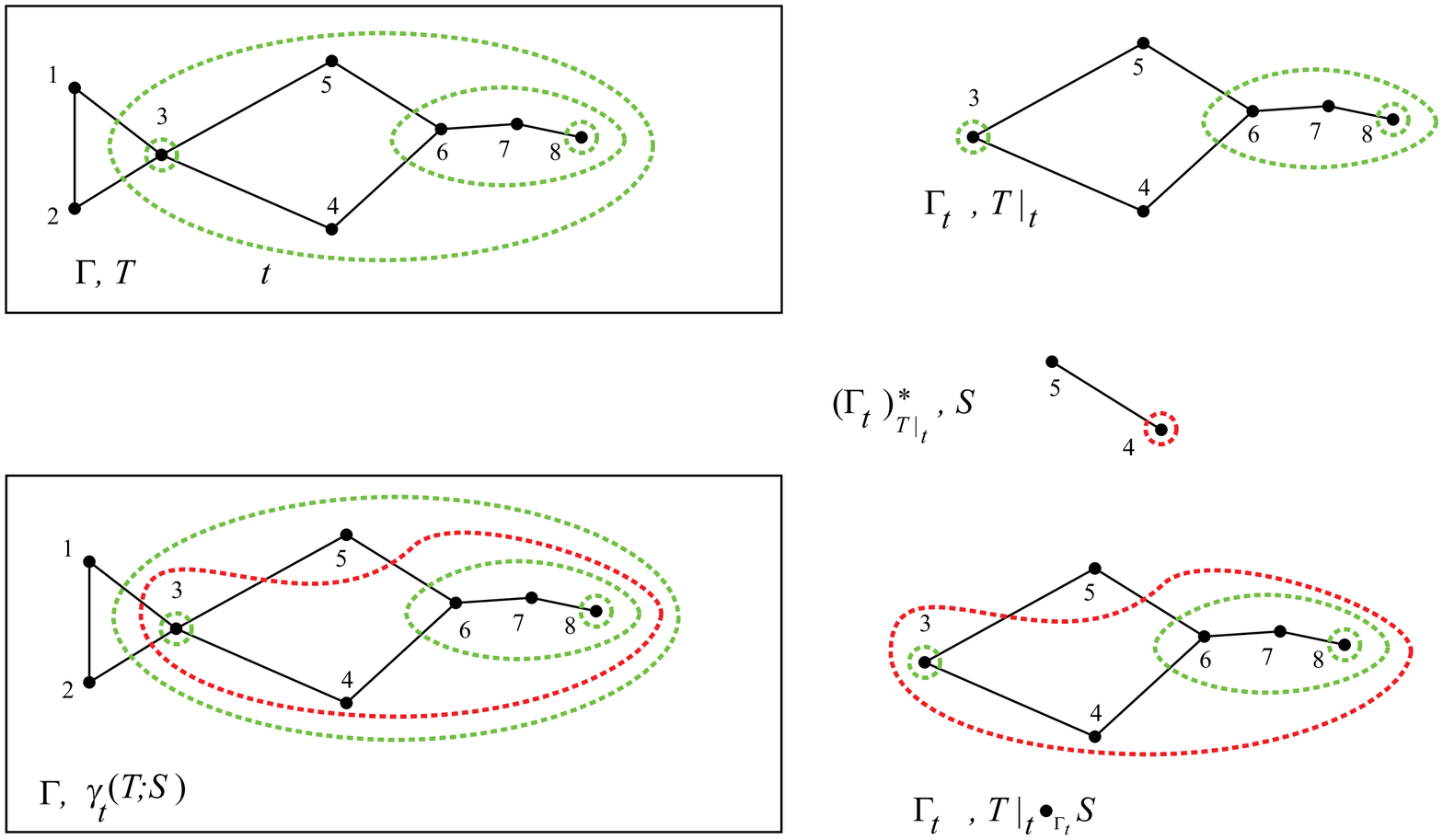}
\medskip

$(3)$ For the complete graph $\Gamma = K_6$. The tubings  $T=\{ \{ 1, 3, 4, 6\},\{ 1, 4\}, t_{K_6} \}$ and $S=\{\{2\}, \{1,2\}\}$ in $(\Gamma _{\{1, 3, 4, 6\}})_{T\vert_{\{1, 3, 4, 6\}}}^* = K_2$, are shown below, followed by the substitution $\gamma _{\{\{1,3,4,6\}\}}(T;S)$:

\begin{table}[h]
\begin{tabular}{ccc}
\includegraphics[width=2.5in]{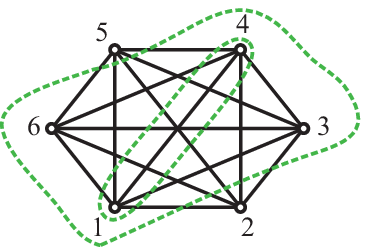}
&&
\includegraphics[width=.75in]{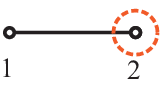}
\end{tabular}
\end{table}
\medskip

\begin{center}
  \includegraphics[width=2.5in]{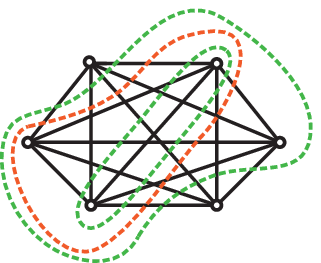}
\end{center}

\bigskip

As there exists a bijection between the set ${\mbox {Tub}(K_6)}$ of tubings of the complete graph $K_6$ and the set
\begin{equation*}{\mbox {\bf ST}_6}:=\{ x:[6]\rightarrow [r]\mid 1\leq r\leq 6\ {\rm and}\ x\ {\rm is\ surjective\}\ },\end{equation*}
we may describe the substitution in terms of surjective maps, as in \cite{LR} .
\medskip

If we draw the surjections determined by $T$ and $S$ as:

\begin{table}[h!]
\begin{tabular}{ccc}
\includegraphics{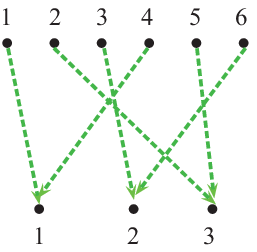}&\qquad\qquad\qquad&
\includegraphics{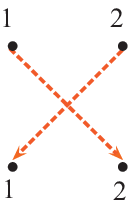}
\end{tabular}
\end{table}
\medskip

then we get that $\gamma_{\{\{1,3,4,6\}\}}(T;S) = (1, 4, 3, 1, 4, 2)$ is given by:
\medskip

\begin{figure}[h!]
\includegraphics{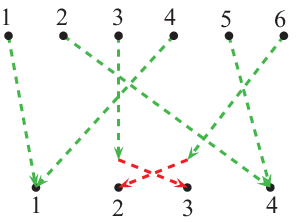}
\end{figure}

\end{example}
\bigskip

\subsection{Associativity of substitution}
\medskip

As in the case of non-symmetric operads, given a tubing $T$ of $\G$, and a family of tubings $S^t\in {\mbox {Tub}((\G_t)_{T\vert_{t}}^*)}$, with $t\in T$, it is possible to define the complete substitution $\gamma (T;\{S^t\}_{t\in T})$.
\medskip

We begin with the following result, whose proof is immediate.
\medskip

\begin{lemma} \label{lemmacomvee} Let $T$ be a tubing in $\G$, and let $t$ and $t\rq $ be a pair of different tubes of $T$. For tubings $S\in {\mbox {Tub}(\G_t)}$ and
$S\rq \in {\mbox {Tub}(\G_{t\rq})}$ such that $S$ is compatible with $T\vert_t$ and $S\rq$ is compatible with $T\circ_tS$, we have that:\begin{enumerate}
\item $S$ is compatible with $T\circ _{t\rq}S\rq$,
\item $(T\circ _t S)\circ _{t\rq }S\rq = (T\circ _{t\rq }S\rq)\circ_t S.$\end{enumerate}\end{lemma}
\medskip

\begin{proposition} \label{propindepofgammafortubes} Let $T$ be a tubing on $\G$. For any pair of tubes $t,t\rq \in T$ and any pair of tubings
$S\in {\mbox {Tub}((\G_t)_{T\vert _t}^*)}$ and $S\rq\in {\mbox {Tub}((\G_{t\rq})_{T\vert _{t\rq}}^*)}$, we have that:\begin{enumerate}
\item $t\rq$ is a tube of $\gamma_t(T;S)$ and $t$ is a tube of $\gamma_{t\rq}(T;S\rq)$,
\item $\gamma_{t\rq}(\gamma _t(T;S);S\rq)=\gamma _t(\gamma_{t\rq}(T;S\rq);S)$.\end{enumerate}\end{proposition}

\begin{proof} The first assertion is evident.
For the second item, we have to consider the following cases:\begin{enumerate}[(i)]
\item if $t\cap t\rq = \emptyset$, then $S\cap t\rq =\emptyset = S\rq \cap t$. Moreover, as there does not exist an edge $\{ v,w\}$ with $v\in t$ and $w\in t\rq$, neither any node in $t$ is connected with a tube in $S\rq$, nor any node in $t\rq$ is connected to a tube in $S$. The result follows immediately.
\item Suppose that $t\cap t\rq \neq \emptyset$. We may suppose, without loss of generality, that $t\rq \subsetneq t$. In this case, ${\mbox {Maxt}((\gamma_t(T;S)\vert _{t\rq})} = {\mbox {Maxt}(T\vert _{t\rq})}$ and $t\rq $ is contained in some maximal tube of $T\vert _t$ (eventually it is a maximal tube). So, the tubes added by $S\rq$ to $ \gamma _t(T;S)$ are contained in $t\rq$ and appear in $\gamma_{t\rq}(T;S\rq)$. In the same way, $S\rq$ does not add tubes to $\gamma _t(T;S)$ or to $S$, which proves the equality.\end{enumerate}
\end{proof}

\begin{definition} \label{defglobalsubst} Let $T$ be a tubing in $\G$ with $k+1$ tubes, and suppose that the set of tubes of $T$ is given a labelling $T=\{ t^0=t_{\G}, t^1,\dots ,t^k\}$.
For any family of tubings $S^i\in {\mbox {Tub}((\G_{t^i})_{T\vert _{t^i}}^*)}$, $0\leq i\leq k$, define the substitution $\gamma (T;S^0,\dots ,S^k)$ of $S^0,\dots ,S^k$ in $T$ recursively as follows:\begin{enumerate}[(a)]
\item For $T=T_{\G}$ and $S\in {\mbox {Tub}(\G)}$, let $\gamma (T;S):= S$,
\item For $T = \{ t^0=t_{\G}, t^1,\dots ,t^k\}$, define

$\gamma (T;S^0,\dots ,S^k):= \gamma _{t^k}(\gamma (T;S^0,\dots ,S^{k-1});S^k).$\end{enumerate}\end{definition}

 Proposition \ref{propindepofgammafortubes} implies that $\gamma (T; S^0,\dots ,S^k)$ is well defined, for each appropriate labelling of the tubes in $T$.

\begin{remark}
Note that the substitution of tubings $S^i$ into a tubing $T$ of a graph $\G$ results in a tubing $T' = \gamma (T; S^0,\dots ,S^k)$ that contains $T$. This implies two results:
 \begin{enumerate}
     \item The tubing $T'$ resulting from substitution into $T$ corresponds to a sub-face of the face corresponding to $T$ on the graph associahedron ${\mathcal K}\G$. This is implied by Theorem~\ref{faces}.
\medskip

     \item The resulting tubing $T'$ is a basis for a topology which refines the topology generated by $T.$ Moreover, the new topology still has the property that nodes connected by edges are path-connected in the topology, as predicted by Corollary~\ref{paths}. However, there may be pairs of nodes in the topology generated by $T$ that are connected by a path whose range is only those two nodes; but for which, in the topology generated by $T'$, any topological path that connects that pair of nodes has a range that includes additional nodes.  For instance, see the nodes 6 and 8 in the following Example~\ref{genex}.
 \end{enumerate}
\end{remark}

\begin{example}\label{genex}
Here is an example on a given graph $\Gamma$ with labelled tubing $T = \{t_{\G}, t^1,t^2,t^3,t^4\} = \{t_{\G}, \{3,4,5,6,7,8\},\{6,7,8\},\{3\},\{8\}\}$  In this case restricting $\G$ to the tube $t^1$ and then forming the reconnected complement with respect to the tubes inside of $t^1$  results in  the path on two nodes. Restricting $\G$ to the tube $t^2$ and then forming the reconnected complement with respect to the tubes inside of $t^2$ also  results in  the path on two nodes. This highlights the need to have a labelling of the tubes in $T$ so as to match the tubings $S^i$ to their
respective destinations.
\begin{center}
   \includegraphics[width=3.5in]{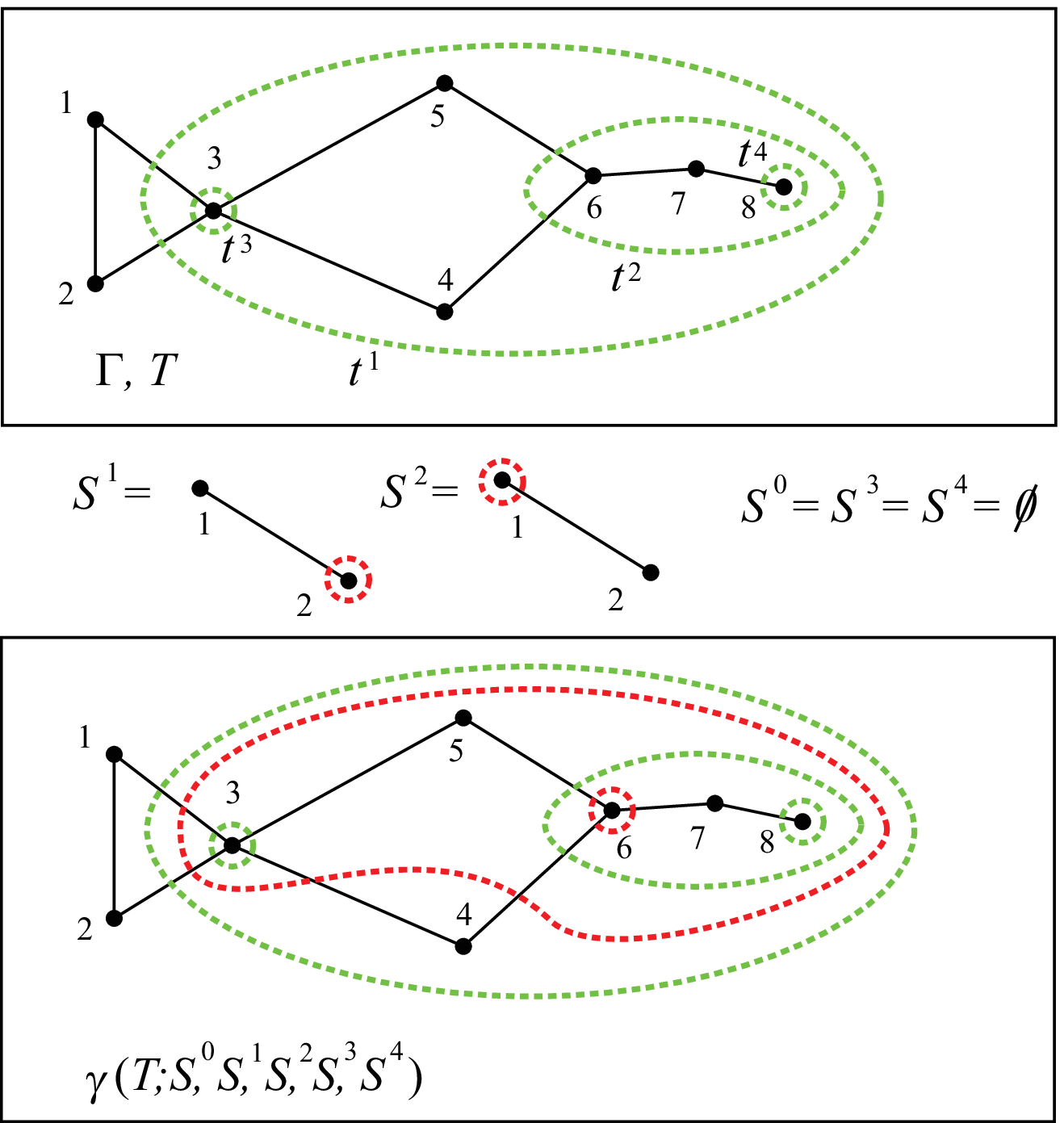}
\end{center}
\end{example}
\newpage

\begin{example}
Consider the complete graph $K_8$ with the tubing as follows:

\begin{table}[h]
\begin{tabular}{ccc}
As tubings in $K_8$&\qquad\qquad\qquad & As surjective maps\\
\includegraphics[width=2in]{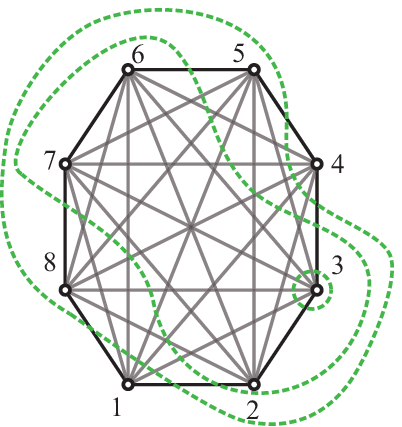}
&\qquad\qquad\qquad&
\includegraphics[width=2in]{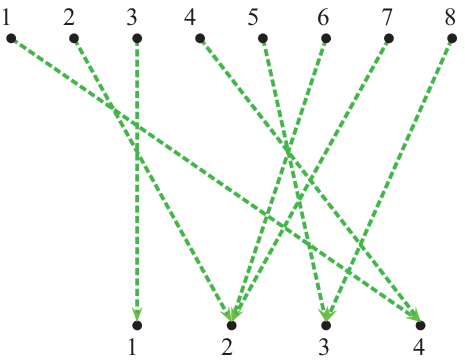}\\
$T=\{t^0,t^1,t^2,t^3\}$\\ $= \{t_{K_8}, \{3\}, \{2,3,6,7\}, \{2,3,5,6,7,8\} \}$&\qquad\qquad\qquad&$x_T=(4,2,1,4,3,2,2,3)$\\
\includegraphics[width=.75in]{subs_p8f2a.eps}
&\qquad\qquad\qquad&
\includegraphics[width=.75in]{subs_p8f2b.eps}\\
$S^0=\{\{2\},\{1,2\}\}$&\qquad\qquad\qquad&$x_{S^0}=(2,1)$\\
\includegraphics[width=1in]{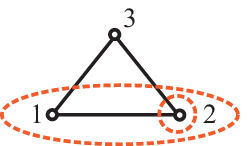}
&\qquad\qquad\qquad&
\includegraphics[width=1in]{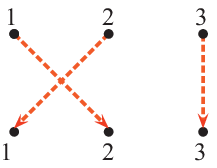}\\
$S^1=T_{\emptyset}$ and $S^2=\{\{2\},\{1,2\},t_{K_3}\}$&\qquad\qquad\qquad&$x_{S^1}=(1)$ and $x_{S^2}=(2,1,3)$\\
\includegraphics[width=.75in]{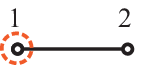}
&\qquad\qquad\qquad&
\includegraphics[width=.75in]{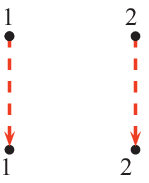}\\
$S^3=\{\{1\},\{1,2\}\}$&\qquad\qquad\qquad&$x_{S^3}=(1,2)$
\end{tabular}
\end{table}

\vfill\eject

the substitution $\gamma (T; S^0,S^1,S^2,S^3)$ is given by:
\medskip

\begin{table}[h]
\begin{tabular}{ccc}
\includegraphics[width=2.5in]{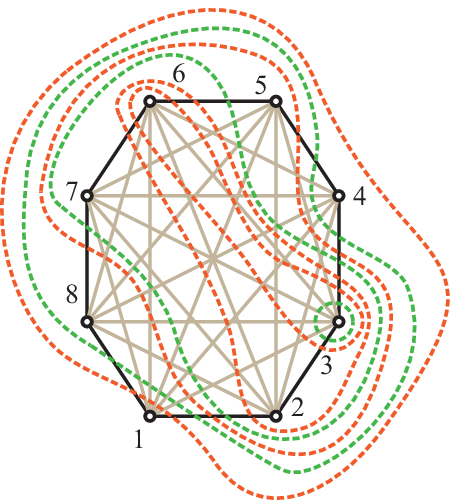}
&
\includegraphics[width=2in]{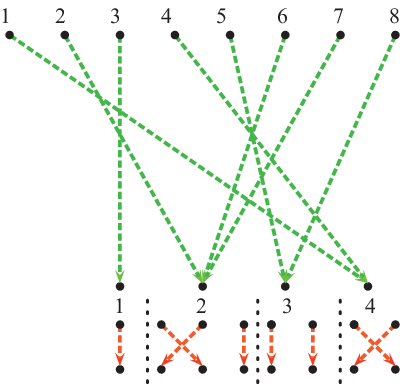}\\
&$x_{\gamma(T;S^0,S^1,S^2,S^3)}= (8,3,1,7,5,2,4,6)$,
\end{tabular}
\end{table}

\begin{align*}
\gamma (&T; S^0, S^1, S^2, S^3)=\\
&\{t_{K_8}, \{3\}, \{3, 6\}, \{2, 3, 6\}, \{2, 3, 6, 7\}, \{2, 3, 5, 6, 7\}, \{2, 3, 5, 6, 7, 8\}, \{2, 3, 4, 5, 6, 7, 8\}\}.\end{align*}
\medskip

\end{example}
\medskip

\begin{notation} \label{stilde} Let $T$ be a tubing of $\G$ and let $s$ be a tube in $\G_T^*$. Let ${\tilde s}$ denotes the tube ${\tilde s} := s\cup \{t^{i_1},\dots ,t^{i_r}\}$, where $t^{i_1},\dots ,t^{i_r}$ are the maximal proper tubes of $T\vert _t$ which are linked to at least one of a union of tubes which comprise $s$ in $\G$. In other words, ${\tilde s}$ is the minimal tube satisfying that $s\subseteq {\tilde s}$ and
$(T, {\tilde s})$ belongs to ${\mbox {Tub}(\G)}$.\end{notation}
\medskip

\begin{theorem} \label{lacompdetubisass} (Associativity of substitution) Let $T$ be a tubing in $\G$ and let $t$ be a tube in $T$. Given a tubing $S\in {\mbox {Tub}((\G_t)_{T\vert _t}^*)}$ and a proper tube $s\in S$, we have that:\begin{enumerate}
\item The graphs $(((\G_t)_{T\vert _t}^*)_s)_{S\vert _s}^*$ and $(\G_{\tilde s})_{\gamma _t(T,S)\vert_{\tilde s}}^*$ are equal.
\item For any tubing $W$ of $(((\G_t)_{T\vert _t}^*)_s)_{S\vert _s}^*$, the tubing $\gamma_s(S; W)$ is a tubing of $(\G_t)_{T\vert _t}^*$, which satisfies that
\begin{equation*}\gamma _t(T,\gamma _s(S,W)) = \gamma _{\tilde s}(\gamma _t(T,S), W).\end{equation*}\end{enumerate}\end{theorem}
\medskip

\begin{proof} For the first assertion, the nodes of the graph $(((\G_t)_{T\vert _t}^*)_s)_{S\vert _s}^*$ are the nodes of $s$ which do not belong to any other tube of $S$, which coincide with the nodes of $(\G_{\tilde s})_{\gamma _t(T,S)\vert_{\tilde s}}^*$. As we work on the subgraph determined by $s$, it suffices to prove the result for $t=t_{\G}$, or equivalently that
\begin{equation*} ((\G_T^*)_s)_{S\vert_s}^*= (\G_{\tilde{s}})_{\gamma_{t_{\G}}(T,S)\vert _{\tilde{s}}}^*,\end{equation*}
for any tubing $S$ compatible with $T$ and $s\in S$.

Clearly, the edges of $((\G_T^*)_s)_{S\vert_s}^*$ are either edges of $(\G_s)$ (which do not belong to any other tube contained in $s$), or the pairs of nodes $u, v$ in $s$ which are linked to a same tube $t\rq$ in $T$.

On the other hand, the edges of $(\G_{\tilde{s}})_{\gamma_{t_{\G}}(T,S)\vert _{\tilde{s}}}^*$, are the edges of $\G$ which have both extremes in $s$ and in no other tube, or the pairs of nodes $u, v$ in $s$ which are linked to a same tube $t\rq$ in $T$ or to a tube $s\rq \in S$, which is contained in $s$. Using that the edges of $((\G_T^*)_s)_{S\vert_s}^*$ are the edges which belong to $s$ (and to no other tube in $S$) or the pairs of nodes in $s$ which are linked to some $s\rq \in S$, with $s\rq \subset s$, we get the result.
\medskip

To prove the second assertion, note that a tube ${\tilde w}$ belongs to $\gamma_s(S; W)$ if either ${\tilde w}\in S$, or it is of the form ${\tilde w} = w \cup \{s^{i_1},\dots ,s^{i_r}\}$, for $w \in W$ and $s^{i_1},\dots ,s^{i_r}$ the maximal tubes in $S\vert _s$ which are linked to $w$.

In a similar way, a tube $t\rq$ in $\G$ belongs to $\gamma_t(T; \gamma_s(S; W))$ whenever it fulfills one of the following conditions:\begin{itemize}
\item $t\rq \in T$,
\item $t\rq = {\tilde w} = w\cup \{t^{j_1},\dots , t^{j_k}\}$ for some $w \in \gamma_s(S; W)$, where $t^{j_1},\dots ,t^{j_k}$ are the maximal tubes of $T\vert _t$ which are linked to $w$.\end{itemize}

So, $\gamma_t(T; \gamma_s(S; W))$ is the collection of tubes $u$ in $\G$ satisfying one of the following conditions:\begin{enumerate}
\item $u\in T$, or
\item $u\subsetneq t$ and $u = {\tilde s\rq}= s\rq \cup \{t^{j_1},\dots , t^{j_k}\}$, for some tube $s\rq \subsetneq s$ in $S$,
\item $u\subset s$ and $u = {\tilde w} = w\cup \{ {\tilde {s}}^{i_1},\dots ,{\tilde {s}}^{i_r}\}\cup \{t^{j_1},\dots ,t^{j_k}\}$ for some tube $w\in W$, where $s^{i_1},\dots , s^{i_r}$ are the tubes
in ${\mbox {Maxt}(S\vert _s)}$ linked to $w$, and $t^{j_1},\dots ,t^{j_k}$ are the maximal tubes of $T\vert _t$ linked to $w$.\end{enumerate}
\medskip

On the other hand, a tube of $\gamma _{\tilde s}(\gamma _t(T,S), W)$ is either a tube of $\gamma _t(T,S)$ or a tube of the form ${\tilde w} = w\cup \{u^{l_1},\dots , u^{l_p}\}$, where $u^{l_1},\dots , u^{l_p}$ are the maximal tubes of $\gamma _t(T,S)$ contained in ${\tilde s}$ which are linked to $w$.

But the elements of ${\mbox {Maxt}(\gamma _t(T,S))}$ contained in ${\tilde s}$ are the maximal tubes of $T\vert _t$ which are not linked to any $s\in S$ and the tubes of the form ${\tilde s\rq}$, for $s\in S$.

Therefore, a tube $u\in \gamma _{\tilde s}(\gamma _t(T,S), W)$ is a tube in $\G$ satisfying one of the following conditions:\begin{enumerate}
\item $u\in T$,
\item $u={\tilde s\rq}$, for some $s\in S$,
\item $u =w\cup \{ {\tilde {s}}^{i_1},\dots ,{\tilde {s}}^{i_r}\}\cup \{t^{j_1},\dots ,t^{j_k}\}$ for some tube $w\in W$, where $s^{i_1},\dots , s^{i_r}$ are the tubes
in ${\mbox {Maxt}(S\vert _s)}$ linked to $w$, and $t^{j_1},\dots ,t^{j_k}$ are maximal tubes of $T\vert _t$ linked to $w$ but not linked to any $s\rq \in S$,\end{enumerate} which ends the proof.
\end{proof}
\medskip

Note that Theorem \ref{lacompdetubisass} implies that the substitution $\gamma$ is associative in the following way:

Let $T = \{ t_{\G}=t^0,t^1,\dots ,t^k\}$ be a tubing in $\G$, and let $S^i=\{ s^{0i},\dots ,s^{l_ii}\}$ be  a family of tubings $S^i\in {\mbox {Tub}((\G_{t^i}})_{T\vert _{t^i}}^*)$, for $0\leq i\leq k$. Suppose that $W^{ji}\in {\mbox {Tub}(\G_{s^{ji}})_{T\vert _{s^{ji}}}^*}$ is a collection of tubings, for each pair $(i,j)$ with $0\leq j\leq i$. We have that:

\begin{align*} \gamma(\gamma (T; &S^0,\dots ,S^k);W^{00},\dots ,W^{l_00},\dots ,W^{l_kk}) = \\
 &\gamma (T; \gamma (\gamma_{t^0}(T,S^0)\vert_{t^0};W^{00},\dots ,W^{l_00}),\dots , \gamma (\gamma_{t^k}(T,S^k)\vert_{t^k};W^{0k},\dots ,W^{l_kk})).\end{align*}
\bigskip

Using the notations introduced in Remarks \ref{rem:restriction} and \ref{rem:restrictionreconnected}, we study the relationship between substitution and restriction.

\begin{lemma} \label{lem:auxrestsubs} Let $\G$ and $\Omega$ be two finite simple connected graphs with the same set of nodes and such that ${\mbox {Edg}(\Omega)} \subseteq {\mbox {Edg}(\G)}$.
For any pair of tubings $T\in {\mbox {Tub}(\G)}$ and $S\in {\mbox {Tub}(\G_T^*)}$, we have that:\begin{enumerate}
\item The graphs $\G_T^*$ and $\Omega _{\hat{T}}^*$ have the same set of nodes and ${\mbox {Edg}(\Omega _{\hat{T}}^*)}\subseteq {\mbox {Edg}(\G_T^*)}$,
\item ${\mbox {rest}_{\Omega}^{\G}(\gamma _{t_{\G}}(T,S))} = \gamma _{t_{\Omega}}({\hat {T}}, {\hat {S}}),$\end{enumerate}
where ${\hat {T}}:={\mbox {rest}_{\Omega}^{\G}(T)}$ and ${\hat {S}}:={\mbox {rest}_{\Omega_{\hat{T}}^{*}}^{\G_T^*}}(S)$.\end{lemma}
\medskip

\begin{proof} The first assertion is easily verified.

A tube ${\hat{t}}\in \gamma _{t_{\Omega}}({\hat {T}}, {\hat {S}})$ is a tube in $\Omega$ which satisfies one of the following conditions:\begin{enumerate}[(a)]
\item ${\hat{t}}$ is a connected component of ${\mbox {rest}_{\Omega}^{\G}({t})}$, for some ${t}\in T$,
\item ${\hat{t}}$ is of the form ${\hat{s}}\cup \{{\hat{ t}}^1,\dots ,{\hat {t}}^r\}$, where ${\hat{s}}$ is a connected component of ${\mbox {rest}_{\Omega_{\hat{T}}^{*}}^{\G_T^*}}(s)$, for some $s\in S$, and ${\hat{t}}^1,\dots ,{\hat{t}}^r$ are the maximal tubings in
${\mbox {rest}_{\Omega}^{\G}(T)}$ which are linked to ${\hat{s}}$ in $\Omega$.\end{enumerate}

On the other hand, given a tube $s\in S$ the tube induced by $s$ in $\gamma _{t_{\G}}(T,S)$ is ${\tilde s}= s\cup \{t^1,\dots ,t^m\}$, where $t^1,\dots ,t^m$ are the maximal proper tubes of $T$ which are linked to $s$. The restriction of ${\tilde s}= s\cup \{t^1,\dots ,t^m\}$ to $\Omega$ is a union of tubes $\{ s_i\}_{1\leq i\leq k}$. Each tube $s_i$, for $1\leq i\leq k$, is of one of the following types:\begin{enumerate}[(a)]
\item $s_i$ is a tube with all its nodes in a maximal tubing $t^{j_i}$ of $T$, and in this case $s_i\in {\mbox {rest}_{\Omega}^{\G}(T)}$, because the maximal tubings of $T$ are not linked,
\item If $s_i$ contains some node in $\G_{T}^*$, then $s_i$ cannot contain two disjoint tubes ${\hat{s}}$ and ${\hat{s}}\rq $ of ${\hat{S}}$ because in this case both tubes must be linked to the same maximal proper tube of $T$, and therefore they cannot be disjoint in $\Omega_{\hat{T}}^{*}$. So, there exist a unique tube ${\hat{s}}\in {\hat{S}}$ such that $s_i = {\hat{s}}\cup \{{\hat{t}}^{j_1},\dots ,{\hat{t}}^{j_r}\}$, where ${\hat{t}}^{j_l}\in {\hat{T}}$ is a connected component of the restriction of a maximal tube
$t^{j_l}$ of $T$ to $\Omega$, for $1\leq l\leq r$, such that ${\hat{t}}^{j_l}$ is linked to ${\hat{s}}$, and $\{t^{j_1},\dots ,t^{j_r}\} \subseteq \{t^1,\dots ,t^m\}$.\end{enumerate}

The last paragraph shows that the tubes of ${\mbox {rest}_{\Omega}^{\G}(\gamma _{t_{\G}}(T,S))}$ are the same than the tubes of $\gamma _{T_{\Omega}}({\hat{T}}, {\hat{S}})$, which ends the proof.\end{proof}
\medskip

The proof of the following result is easily obtained from Lemma \ref{lem:auxrestsubs}, by applying a recursive argument on $k$.

\begin{proposition} \label{lem:restrictiongamma} Let $\G$ and $\Omega $ be two graphs with the same set of nodes satisfying that ${\mbox {Edg}(\Omega)}\subseteq {\mbox {Edg}(\G)}$. For any $T\in {\mbox {Tub}(\G)}$, any tube $t\in T$ and any tubing $S\in {\mbox {Tub}(\G_{T\vert_t}^*)}$ , we have that:
\begin{equation*}\gamma ({\mbox {res}_{\Omega}^{\G}}(T); {\hat{S}}^0,\dots ,{\hat{S}}^k)= {\mbox {res}_{\Omega}^{\G}}(\gamma _t(T, S)),\end{equation*}
where the tube $t$ induces a tubing ${\mbox {res}_{\Omega}^{\G}}(t)=\{ t_1,\dots ,t_k\}$, with $t_i\cap t_j = \emptyset$ for $i\neq j$, and ${\hat{S}}^i$ is the tubing induced by $S$ on the reconnected complement $(\Omega _{t_i})_{{\mbox {res}_{\Omega}^{\G}}(T)\vert_{t_i}}^*$.\end{proposition}
\medskip

\begin{proposition}\label{propgenerators} For any connected graph $\G$, a tubing $T\in {\mbox {Tub}(\G)}$ may be obtained from $(\G, \{t_{\G}\})$ applying substitutions of type $\gamma _{t_{\G}}(\ , \{ t\})$, where $\{ t\}$ denotes the tubing whose unique tubes are $t_{\G}$ and $t$.\end{proposition}
\medskip

\begin{proof} Let $T$ be a tubing of $\G$ with ${\overline T} =\{ t^1,\dots , t^k\}$. We use a recursive argument on $k$.

If $k=1$, then $T = \gamma_{t_{\G}} (t_{\G}, \{ t^1\})$.

If $k\geq 2$, we may suppose that $t^k$ is a maximal tube in $T$ (we may change the order of the $t^j$\rq s to get it). By the recursive argument, $T\rq :=\{t^1,\dots ,t^{k-1}\}$ is obtained as $\gamma_{t_{\G}} (\dots (\gamma_{t_{\G}} (t_{\G}, \{t^{i_1}\}),\dots , \{t^{i_{k-1}}\}))$.

As $t^k$ is maximal in $T$, then $\{ t^k\}$ restricts to a tubing $S^k$ with a unique tube in $\G_{T\rq}^*$, and it is not difficult to see that $\gamma (T\rq , S^k)= T$.\end{proof}
\medskip

\begin{example}The tubing $T$
\begin{figure}[h!]
\includegraphics[width=50mm,scale=0.8]{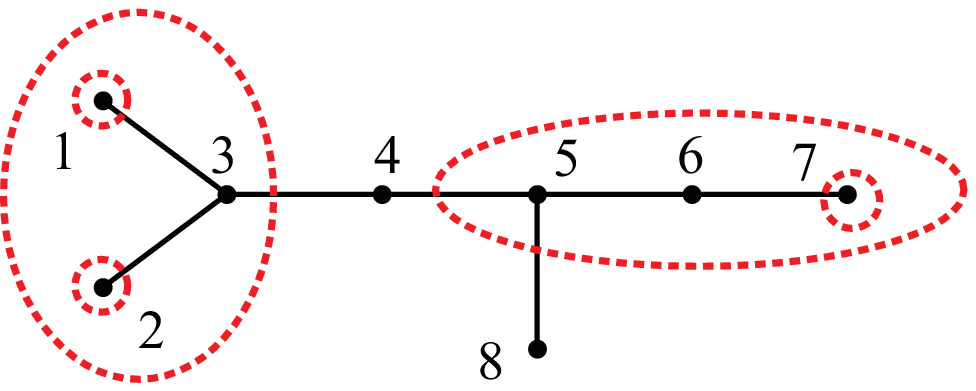}
\end{figure}

is obtained as
\begin{equation*}T = \gamma_{t_{\G}}(\gamma_{t_{\G}}(\gamma _{t_{\G}}(\gamma _{t_{\G}}(\gamma _{t_{\G}}(t_{\G}, \{1\}), \{2\}),\{7\}),\{3\}), \{5,6\}).\end{equation*}
\end{example}
\bigskip

\section{Algebraic description of graph associahedra and relation with non-symmetric operads}
\bigskip

We give a complete description of ${\mbox {\bf Tub}}$ as freely spanned by the set of all finite connected simple graphs and partial operations $\circ _{\G, t}$, where $t$ is a connected subgraph of $\G$, defined by substitution. The boundary map of M. Carr and S. Devadoss graph associahedra is also defined as a natural derivation for this structure.

Finally, we show that non-symmetric operads (see \cite{MSS} or \cite{LV}) and permutads (see \cite{LR}) are particular examples of substitution of graph associahedra.
\medskip

\subsection{Graph associahedra described by operations and relations}
\medskip

 Let us define a partial composition on ${\mbox {\bf Tub}}$.
 \medskip
\begin{definition}\label{partialcomp} Let $\G$ be a graph with ${\mbox {Nod}(\G)} =[n]$, and let ${\mbox{Nod}(t)} = \{ i_1<\dots <i_{\vert t\vert}\}$ be a tube in $\G$, where $\vert t\vert$ denotes the number of nodes in $t$.
\begin{enumerate}[(a)]
\item The element $\sigma_t\in \Sigma _n$ is the permutation whose image is
\begin{equation*}\sigma_t:=(i_1,\dots ,i_{\vert t\vert}, j_1,\dots ,j_{n-\vert t\vert}),\end{equation*}
where ${\mbox {Nod}(\G_t^*)}= [n]\setminus {\mbox {Nod}(t)}=\{ j_1<\dots <j_{n-\vert t\vert}\}$.
\item For any permutation $\Sigma\in \Sigma_n$, let ${inv(\sigma)}$ denote the number of pairs $1\leq i < j\leq n$ satisfying that $\sigma^{-1}(j) <\sigma^{-1}(i)$ and that the pair $(i,j)$ is an edge in $\G$.
The {\it signature} of a graph $\G$ is the map which assigns to any permutation $\sigma\in \Sigma_n$ the integer $(-1)^{\vert {inv}(\sigma)\vert}$, we denote it by ${\mbox {sgn}^{\G}(\sigma)}$.
\item Let $S$ be a tubing of $\G_t^*$ such that $t$ is not linked to any proper tube $s$ of $S$ in $\G$. Define the integer $\alpha (t,S)$ as follows:\begin{itemize}
\item $\alpha(t,S):=-1$, when the minimal node of $t$ is smaller than the minimal node of the disjoint union $\displaystyle{\bigcup_{s\in \text{Maxt}(S)}s}$,
\item $\alpha(t,S):= 1$, when the minimal node of $t$ is larger than the minimal node of some maximal tube $s$ of $S$.\end{itemize}
\item The binary operation $\circ _{(\G, t)}$ is partially defined on ${\mbox {\bf Tub}}$, as follows:
\begin{equation*} S\circ _{(\G, t)} W:= \begin{cases} \alpha(t,W)\cdot (-1)^{\vert t\vert} \gamma _{t_{\G}}(T_{\G}\circ _tS, W),&{\rm for}\ t\ {\rm not\ linked\ to\ any\ proper\ tube\ of}\ W,\\
\gamma _{t_{\G}}(T_{\G}\circ _tS, W),&{\rm for}\ t\ {\rm linked\ to\ a\ proper\ tube\ of}\ W,\end{cases}\end{equation*} for any pair of tubings $S\in {\mbox{Tub}(\G_t)}$ and $W\in {\mbox {Tub}(\G_t^*)}$, where $T_{\G}\circ _tS$ is the tubing of $\G$ induced by $S\in {\mbox {Tub}(\G_t)}$ introduced in Definition \ref{defn:inductions}. That is, a tube $w\in T_{\G}\circ _tS$ is either $t_{\G}$ or a tube in $S$.\end{enumerate}\end{definition}

\vfill\eject

\begin{example} For

\begin{figure}[h!]
\includegraphics[scale=0.5]{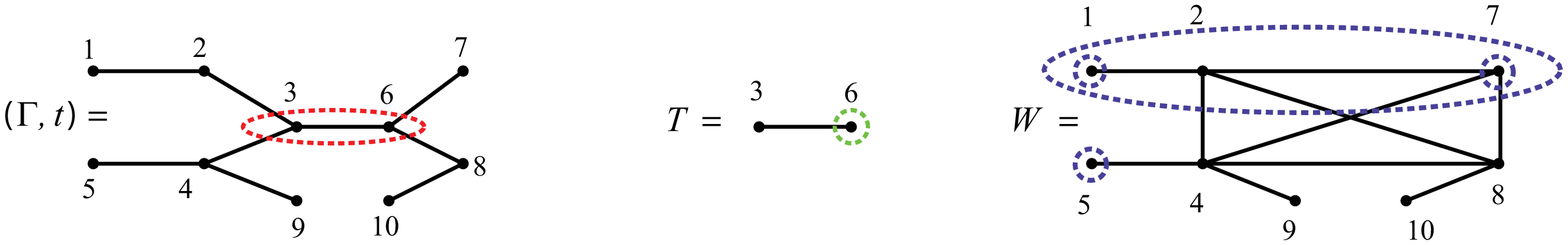}
\end{figure}

since $t$ is linked to a proper tube of $W$, we get that $T\circ_{(\G,t)}W$ is

\begin{figure}[h!]
\includegraphics[scale=0.5]{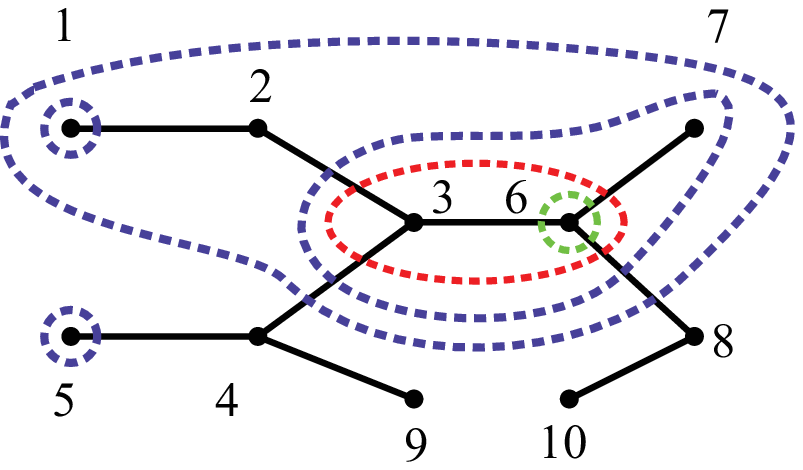}
\end{figure}

\end{example}
\medskip

\begin{proposition}\label{propscirc} The operations $\circ _{(\G, t)}$ satisfy the following relations:\begin{enumerate}
\item For two tubes $t$ and $t\rq $ in a graph $\G$, which are not linked, we get that:
\begin{equation*}T_2\circ _{(\G,t\rq)}(T_1\circ _{(\G_{t\rq}^*, t)} S) = {\mbox {max}\{\alpha (t,S); \alpha(t\rq, S)\}} T_1\circ_{(\G,t)}(T_2\circ_{(\G_t^*, t\rq)} S),\end{equation*}
for $T_1\in {\mbox {Tub}(\G_t)}$, $T_2\in {\mbox {Tub}(\G_{t\rq})}$ and $S\in {\mbox {Tub}(\G_{t, t\rq}^*)}$, where ${\mbox {max}\{\alpha (t,S); \alpha(t\rq, S)\}}$ denotes the maximal integer between $\alpha (t,S)$ and $\alpha (t\rq ,S)$.
\item For two tubes $t\rq \subsetneq t$ in a graph $\G$,
\begin{equation*}(T_2\circ_{(\G_t,t\rq)} T_1)\circ_{(\G,t)} S = T_2\circ _{(\G,t\rq)} (T_1\circ _{(\G_{t\rq}^*,{\tilde t})} S),\end{equation*}
for $T_1\in {\mbox {Tub}((\G_{t})_{t\rq}^*)}$, $T_2\in {\mbox {Tub}(\G_{t\rq})}$ and $S\in {\mbox {Tub}(\G_{t}^*)}$, where $\tilde t$ denotes the tube induced by $t$ in $\G_{t\rq}^*$ .\end{enumerate}
\end{proposition}
\medskip

\begin{proof} For the first point, recall that $t$ and $t\rq $ are not connected, so $\G_{\{t,t\rq\}}^* = (\G_t^*)_{t\rq}^*=(\G_{t\rq}^*)_t^*$. It is immediate to verify that:
\begin{equation*}\gamma_{t_{\G}}({\tilde {T}}_2, \gamma _{t_{\G_{t\rq}^*}}({\tilde {T}}_1, S)) = \gamma_{t_{\G}}({\tilde {T}}_1, \gamma _{t_{\G_{t}^*}}({\tilde {T}}_2, S)).\end{equation*}

Note that:
\begin{align*}
T_2\circ _{(\G,t\rq)}(T_1\circ _{(\G_{t\rq}^*, t)} S) =& \alpha(t\rq, \gamma_{t_{\G_{t\rq}^*}}(T_1,S))\cdot \alpha(t,S) \gamma_{t_{\G}}({\tilde {T}}_2, \gamma _{t_{\G_{t\rq}^*}}({\tilde {T}}_1, S)),\\
T_1\circ _{(\G,t)}(T_2\circ _{(\G_{t}^*, t\rq)} S) = &\alpha(t, \gamma_{t_{\G_{t}^*}}(T_2,S))\cdot \alpha(t\rq,S) \gamma_{t_{\G}}({\tilde {T}}_1, \gamma _{t_{\G_{t}^*}}({\tilde {T}}_2, S)).\end{align*} But, we have that: \begin{itemize}
\item if $\alpha(t,S)= -1=\alpha(t\rq, S)$, then $\alpha (t\rq, \gamma_{t_{\G_{t\rq}^*}}(T_1,S))\cdot \alpha (t, \gamma_{t_{\G_{t}^*}}(T_2,S)) =1$,
\item if $\alpha (t,S) = 1 = \alpha(t\rq, S)$, then $\alpha (t\rq, \gamma_{t_{\G_{t\rq}^*}}(T_1,S)) =1= \alpha (t, \gamma_{t_{\G_{t}^*}}(T_2,S))$,
\item If $\alpha (t,S) \neq \alpha(t\rq, S)$, then $\alpha (t, \gamma_{t_{\G_{t}^*}}(T_2,S))=\alpha (t,S)$ and $\alpha (t\rq, \gamma_{t_{\G_{t\rq}^*}}(T_1,S)) = \alpha(t\rq, S)$,\end{itemize}
which implies that
\begin{equation*}\alpha(t\rq, \gamma_{t_{\G_{t\rq}^*}}(T_1,S))\cdot \alpha(t,S) =  {\mbox {max}\{\alpha (t,S); \alpha(t\rq, S)\}}\cdot \alpha(t, \gamma_{t_{\G_{t}^*}}(T_2,S))\cdot \alpha(t\rq ,S),\end{equation*}
and the first equality of the proposition follows immediately.
\bigskip

Suppose now that $t\rq \subsetneq t$. Note that a connected subgraph $w$ in $\G$ is a tube of $\gamma_{t_{\G}}(T_2,\gamma_{t_{\G_{t\rq}^*}}(T_1, S)$ if, and only if, it fulfills one of the following conditions :\begin{itemize}
\item $w$ is a tube of $T_2$,
\item $w$ a tube of $T_1$ which has no node connected to $t\rq$,
\item $w= {\tilde w}\cup t\rq$, with ${\tilde w}\in T_1$ is a tube linked to $t\rq$,
\item $w$ is a tube of $S$ which is not linked to $t$,
\item $w={\tilde w}\cup t$, with ${\tilde w}\in S$ is a tube linked to $t$.\end{itemize}

It is not difficult to see that the tubes of $\gamma_{t_G}(\gamma_{t_{\G_t}}(T_2, T_1), S)$ are exactly the connected subgraphs of $\G$ which satisfy the same conditions. So, we get the second equality.
\end{proof}
\medskip

Propositions \ref{propgenerators} and \ref{propscirc} imply the following result.

\begin{theorem}\label{th:generators} Let ${\mbox {CGraph}}$ be the set of all graded simple connected finite graphs. The vector space ${\mbox {\bf Tub}}$, equipped with the partially defined binary operations $\circ _{(\G, t)}$, for any $\G\in {\mbox {CGraph}}$ and any tube $t$ in $\G$, is the free object spanned by the set ${\mbox {CGraph}}$ and the products $\circ _{(\G, t)}$, under the relations described in Proposition \ref{propscirc}.\end{theorem}
\bigskip

Let us introduce a differential structure on ${\mbox {\bf Tub}}$.

\begin{notation} Let $T$ be a tubing of a graph $\G$, we denote by $\vert\vert T\vert\vert$ the dimension of the face of the polytope ${\mathcal K}\G$ determined by $T$. If $T$ has $k$ tubes and $\G$ has $n$ nodes, then
$\vert\vert T\vert\vert = n{-}k$.\end{notation}

\begin{definition} \label{defn:boundarymap} Let $\G$ be a finite connected simple graph, define the map $\partial:\K[{\mbox {Tub}(\G)}]\longrightarrow \K[{\mbox {Tub}(\G)}]$, where
$\K[{\mbox {Tub}(\G)}]$ denotes the vector space spanned by the set of tubings of $\G$, as the unique $\K$-linear endomorphism satisfying:\begin{enumerate}
\item $\partial (T_{\G})=\sum _{t\subsetneq \G} (-1)^{\vert t\vert} {\mbox {sgn}^{\G}(\sigma _t)} \{t_{\G},t\},$

where the sum is taken over all the tubes $t$ in $\G$ different from the universal tube $t_{\G}$, while the permutation $\sigma _t$ and the map ${\mbox {sgn}^{\G}}$ are the ones introduced at Definition \ref{partialcomp}.
\item  Define, for any tube $t$ in $\G$ and any pair of tubings $T\in {\mbox {Tub}(\G_t)}$ and $S\in {\mbox{Tub}(\G_t^*)}$,
\begin{equation*}\partial (T\circ_{(\G,t)}S) = \partial(T)\circ_{(\G,t)}S + (-1)^{\vert\vert T\vert\vert} T\circ_{(\G,t)}\partial (S).\end{equation*}\end{enumerate}\end{definition}

Theorem \ref{th:generators} shows that there exists a unique linear map $\partial$ satisfying both conditions. Note that $\vert\vert \partial (T)\vert\vert = \vert\vert T\vert\vert {- 1}$. Also note that if $T$ has the maximal number of tubes (so $T$ has dimension 0) then $\partial (T) =0$.

Here is an example of calculating the boundary recursively via Definition~\ref{defn:boundarymap}.

\includegraphics[width=5in]{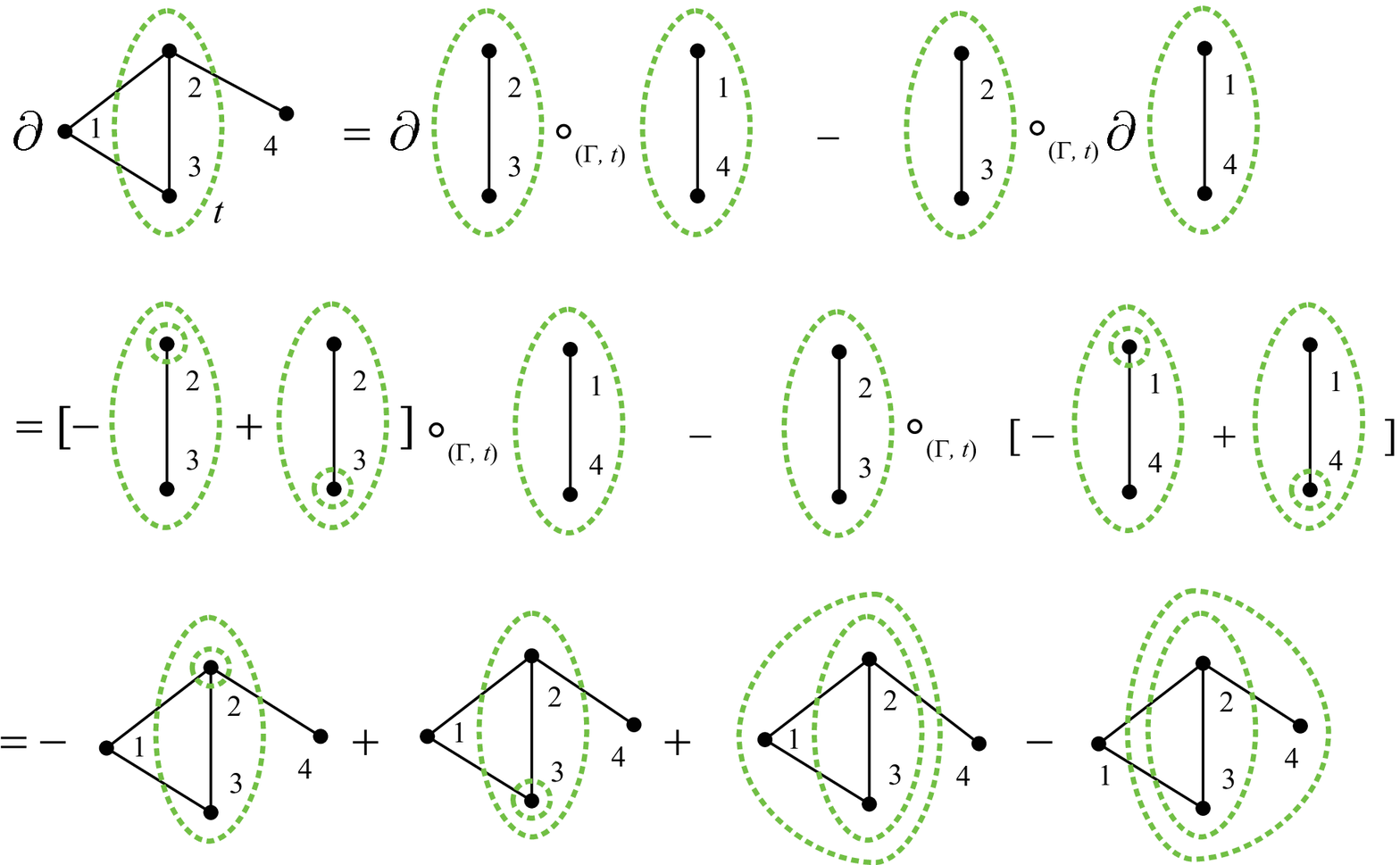}

\begin{proposition} \label{prop:boundary} The homomorphism $\partial $ satisfies that the composition $\partial \circ \partial = 0$.\end{proposition}
\medskip

\begin{proof} Let $T\in {\mbox {Tub}(\G)}$ be a $k$-tubing. We proceed by recursion on $k$.

For $k=1$, we have that $T = T_{\G} $ is the tubing whose unique tube is the universal one. We have that:
\begin{equation*}\partial ^2(T_{\G}) = \sum _{t\subsetneq \G} (-1)^{\vert t\vert} {\mbox {sgn}^{\G}(\sigma _t)} \partial (\{t\}).\end{equation*}
But $\{t\} = T_{\G_t}\circ_{(\G,t)} T_{\G_t^*}$, so
\begin{align*}
\partial (\{t\}) = &\sum _{t\rq\subsetneq t} (-1)^{\vert t\rq\vert} {\mbox {sgn}^{\G_t}(\sigma_{t\rq})} \{t\rq\} \circ_{(\G, t)} T_{\G_t^*} +\\
\qquad & (-1)^{\vert t\vert -1} \sum _{t\rq \subsetneq \G_t^*} (-1)^{\vert t\rq\vert } {\mbox {sgn}^{\G_t^*}(\sigma _{t\rq})} \{t\}\circ_{(\G,t)} \{t\rq\}.\end{align*}

So, in $\partial ^2(T_{\G})$ we get two different class of $3$-tubings\begin{enumerate}
\item the tubings $\{t, t\rq\}$, where $t$ and $t\rq$ are not linked,
\item the tubings $\{t, t\rq\}$, where $t\rq \subsetneq t\}$.\end{enumerate}

Let us compute the coefficient of the tubing $\{ t, t\rq\}$ in $\partial ^2(T_{\G})$ for both cases:

$(1)$ In the first case, we get that the coefficient of $\{ t, t\rq\} = \gamma_{t_{\G}}(T_{\G_t}, \{t\rq\})= \gamma_{t_{\G}}(T_{\G_{t\rq }}, \{t\})$ is:\begin{enumerate}[(i)]
\item $(-1)^{\vert t\vert}{\mbox {sgn}^{\G}(\sigma _t)} (-1)^{\vert t\vert -1} (-1)^{\vert t\rq\vert} {\mbox {sgn}^{\G_t^*}(\sigma _{t\rq})}(-1)^{\vert t\vert}\alpha(t, \{t\rq\})$, coming from $\partial (\{t\})$,
\item $(-1)^{\vert t\rq\vert}{\mbox {sgn}^{\G}(\sigma _{t\rq})} (-1)^{\vert t\rq \vert -1} (-1)^{\vert t\vert} {\mbox {sgn}^{\G_{t\rq}^*}(\sigma _{t})}(-1)^{\vert t\rq\vert}\alpha(t\rq, \{t\})$, coming from $\partial (\{t\rq\})$.\end{enumerate}

Is is immediate to verify that, when $t$ and $t\rq$ are not linked in $\G$, we get that ${\mbox {sgn}^{\G}(\sigma _t)} {\mbox {sgn}^{\G_t^*}(\sigma _{t\rq})}={\mbox {sgn}^{\G}(\sigma _{t\rq})} {\mbox {sgn}^{\G_{t\rq}^*}(\sigma _{t})}$.
\medskip

So, the coefficient of $\{ t, t\rq\} $ in $\partial ^2(T_{\G})$ is:
\begin{align*}
{\mbox {sgn}^{\G}(\sigma _t)}{\mbox {sgn}^{\G_t^*}(\sigma _{t\rq})}&[(-1)^{3\vert t\vert+ \vert t\rq\vert-1}\alpha(t, \{t\rq\}) + (-1)^{3\vert t\rq\vert+ \vert t\vert-1}\alpha(t\rq, \{t\})]= \\
\qquad &{\mbox {sgn}^{\G}(\sigma _t)}{\mbox {sgn}^{\G_t^*}(\sigma _{t\rq})} (-1)^{\vert t\vert+ \vert t\rq\vert-1}[\alpha(t, \{t\rq\})+\alpha(t\rq, \{t\})] = 0,\end{align*} because $\alpha(t, \{t\rq\}) = - \alpha(t\rq, \{t\})$.
\medskip

$(2)$ When $t\rq \subsetneq t$, the tube $\{ t, t\rq\}$ appears twice in the computation of $\partial ^2(T_{\G})$:\begin{enumerate}[(i)]
\item in $\partial (\{ t\}) =\partial ( \gamma _{t_{\G}}(T_{\G_t}, t_{\G_t^*})$, its coefficient is
\begin{equation*}(-1)^{\vert t\vert}{\mbox {sgn}^{\G}(\sigma _t)}(-1)^{\vert t\rq\vert}{\mbox {sgn}^{\G_t}(\sigma _{t\rq})},\end{equation*}
\item in $\partial (\{ t\rq\}) =\partial ( \gamma _{t_{\G}}(T_{\G_{t\rq}}, t_{\G_{t\rq}^*})$, its coefficient is
\begin{equation*}(-1)^{\vert t\rq\vert}{\mbox {sgn}^{\G}(\sigma _{t\rq})}(-1)^{\vert t\rq\vert-1}(-1)^{\vert t\vert -\vert t\rq\vert}{\mbox {sgn}^{\G_{t\rq}^*}(\sigma _{t-t\rq})}.\end{equation*}\end{enumerate}
Again, it is easily seen that ${\mbox {sgn}^{\G}(\sigma _t)}{\mbox {sgn}^{\G_t}(\sigma _{t\rq})} = {\mbox {sgn}^{\G}(\sigma _{t\rq})}{\mbox {sgn}^{\G_{t\rq}^*}(\sigma _{t-t\rq})}$.

\noindent So, the coefficient of $\{ t, t\rq\}$ in $\partial ^2(T_{\G})$ is
\begin{equation*}{\mbox {sgn}^{\G}(\sigma _t)}{\mbox {sgn}^{\G_t}(\sigma _{t\rq})} [(-1)^{\vert t\vert + \vert t\rq\vert}+ (-1)^{\vert t\rq\vert+ \vert t\vert -1}] = 0,\end{equation*}
and we get that $\partial ^2(T_{\G}) = 0$.

When $k\geq 2$, we get that $T = (-1)^{\vert t\vert }\alpha(t,S) T\vert _t\circ _{\G, t} T_t^*$, for some tube $t\in T$, $t\neq t_{\G}$.

The number of tubes of $T\vert_t$ and the number of tubes of $T_t^*$ are both smaller than $k$, so a recursive argument on $\vert\vert T\vert\vert$ and the second condition of Definition \ref{defn:boundarymap} imply that $\partial^2(T)=0$.\end{proof}
\bigskip

\subsection{Non-symmetric operads and permutads}
\medskip

We begin by showing that non-symmetric operads (see \cite{DK} and \cite{LV}) and the permutads, defined in \cite{LR}) are described by the restriction of the data $({\mbox {\bf Tub}}, \circ _{(\G,t)})$ to the vector space spanned by the tubings $T\in {\mbox {Tub}(L_n)}$ of the linear graphs, respectively by the tubings $T\in {\mbox {Tub}(K_n)}$ of the complete graphs, for $n\geq 1$.
\medskip

\begin{notation}\label{permutations} We denote by $1_n$ the identity in the group $\Sigma _n$ of permutations of $n$ elements and by $\cdot $ the product in $\Sigma _n$. Given permutations $\sigma\in \Sigma _n$ and $\tau \in \Sigma _m$, the permutation $\sigma \times \tau \in \Sigma _{n+m}$ is the concatenation of permutations, defined by:
\begin{equation*}\sigma \times \tau (i):=\begin{cases} \sigma(i),&{\rm for}\ 1\leq i\leq n,\\
\tau(i-n)+n,&{\rm for}\ n+1\leq i\leq n+m.\end{cases}\end{equation*}.\end{notation}
\medskip

There exists a simple way to define non-symmetric operads and permutads, by generators and relations.

\begin{definition}\label{defnnonsymperm} \begin{enumerate}\item A {\it non symmetric operad} is a graded vector space $V=\bigoplus _{n\geq 0}V_n$ equipped with linear maps $\circ _i: V_n\otimes V_m\longrightarrow V_{n+m-1}$ satisfying the following conditions:\begin{enumerate}[(a)]
\item $(x\circ _i y)\circ _j z = (x\circ _{j-\vert y\vert +1} z)\circ_{i}y$, for $i < j{-}\vert y\vert +1$,
\item $x\circ _i (y\circ _j z) = (x\circ _i y)\circ _{j+i} z$,\end{enumerate}
for $x,y,z$ in $V$.
\item A {\it permutad} is a graded vector space $V=\bigoplus _{n\geq 0}V_n$ equipped with linear maps $\circ _{\sigma}: V_n\otimes V_m\longrightarrow V_{n+m-1}$, for $\sigma$ an $(n,m)$ shuffle, satisfying the following condition:\begin{enumerate}[(a)]
\item $(x\circ _{\sigma} y)\circ _{\tau} z = x\circ _{\delta} (y\circ_{\gamma}z)$,\end{enumerate}
whenever we have $\tau \cdot (\sigma \times 1_{\vert z\vert} )=  \delta \cdot (1_{\vert x\vert}\times \gamma )$.\end{enumerate} for $x,y,z$ in $V$.\end{definition}
\medskip

\begin{theorem}\label{teo:permutads} The restriction of the products $\circ _{\G,t}$ to graphs of type $\G = L_n$, for $n\geq 1$, define non symmetric operads. In a similar way, the restriction of the products $\circ _{\G,t}$ to graphs of type $\G = K_n$, for $n\geq 1$, define permutads.\end{theorem}
\medskip

\begin{proof} Note that a tube $t$ on a linear graph $L_n$ is of type $t_{i,r}= \{i+1,\dots ,i+r\}$. So, when the graph $\G$ is $L_n$ for some $n\geq 1$, the unique products $\circ _{\G, t}$ are of the form $\circ _{L_n, t_{i,r}}$. If we fix $n$ and $m$, then the restriction of the $\circ _{\G,t}$ to the linear graphs is given by:
\begin{equation*}\circ _{(L_{n+m}, t_{i,m})}: V_n\otimes V_m\longrightarrow V_{n+m}.\end{equation*}
So, for $n$ and $m$ fixed, we get that $\circ _i = \circ _{(L_{n+m}, t_{i,m})}$.

The first relation of non-symmetric operads in Definition \ref{defnnonsymperm} is just the translation to this context of relation $(1)$ of Proposition \ref{propscirc}, while the second relation corresponds to relation $(2)$ of the same Proposition.
\medskip

For permutads, note that for $\G = K_n$, we have that:\begin{enumerate}[(i)]
\item A tube $t$ in $K_n$ is defined by a sequence $1\leq i_1<\dots <i_r \leq n$ of $r$ integers. So, a tube $t=\{ i_1<\dots <i_r \}$ of $K_n$ is identified with the unique $(r, n{-}r)$-shuffle $\sigma _t$ such that $\sigma _t(j)= i_j$, for $1\leq j\leq r$. This point gives the identification of the product $\circ _{K_n,t}$ with $\circ _{\sigma _t}$.
\item As any pair of nodes in $K_n$ is connected, there do not exist tubes $t$ and $t\rq$ in any tubing $T$ of $K_n$ such that $t\cap t\rq = \emptyset $. So, the relation $(1)$ of Proposition \ref{propscirc}, does not apply when $\G = K_n$. Relation $(2)$ of Proposition \ref{propscirc} gives the relations of the products $\circ _{\sigma}$ of Definition \ref{defnnonsymperm}.\end{enumerate}
\end{proof}

\bigskip

\bigskip

\section{Algebraic constructions on ${\mbox {\bf DTub}}$}
\bigskip

Let ${\mbox {\bf DTub}}$ denote the $\K$-vector space spanned by the tubings on all simple finite graphs (not necessarily connected), equipped with a total order on the set of nodes. We have studied in the first sections a way to define M. Carr and S. Devadoss graph associahedra as a free differential object spanned by the set of all finite connected simple graphs, for some algebraic structure whose operations are given by connected subgraphs of a graph with the natural relations satisfied by these subgraphs. In the present section we want to show that ${\mbox {\bf DTub}}$ is the free associative trialgebra spanned by ${\mbox {\bf Tub}}$,
which completes the algebraic description of graph associahedra.
\bigskip

\subsection{${\mbox {\bf DTub}}$ as the free ${\mbox {Trias}\rq}$ algebra spanned by ${\mbox {\bf Tub}}$}
\medskip

J.-L. Loday introduced (see \cite{Lod}) associative dialgebras, as the enveloping algebras of Leibniz algebras. The non-symmetric operad of associative dialgebras is completely described on the space spanned by the vertices of the standard simplexes. We know, see \cite{CD}, that the standard simplex of dimension $n$ is the graph associahedra ${\mathcal K}C_{n+1}$ of the simple graph with $n+1$ nodes and no edge. In \cite{Cha}, F. Chapoton introduced the differential operad ${\mbox {Trias}\rq}$, whose free object on one element is the vector space spanned by all the faces of standard
simplexes, equipped with the usual boundary map, and three binary products.

\begin{definition} \label{triass} (F. Chapoton) A {\it ${\mbox {Trias}\rq}$ algebra} is a graded differential vector space $(V, d )$ equipped with three associative products $\vdash$, $\times$ and $\dashv$, which satisfy the following relations:\begin{enumerate}[(i)]
\item $(x\vdash y)\dashv z =  x\vdash (y\dashv z)$,
\item $(x\dashv y)\dashv z= x\dashv (y\vdash z)$,
\item $(x\dashv y)\vdash z = x\vdash (y\vdash z)$,
\item $(x\vdash y)\times z= x\vdash (y\times z)$,
\item $(x\times y)\dashv z= x\times (y\dashv z)$,
\item $(x\dashv y)\times z = x\times (y\vdash z)$,
\item $(x\times y)\vdash z = 0 = x\dashv (y\times z)$,
\item $d(x\times y) = d(x)\times y + (-1)^{\vert x\vert +1} x\times d(y) + (-1)^{\vert x\vert} (x\dashv y - x\vdash y)$,
\item $d(x\vdash y) = d(x)\vdash y + (-1)^{\vert x\vert} x\vdash d(y)$,
\item $d(x\dashv y) = d(x)\dashv y + (-1)^{\vert x\vert} x\dashv d(y)$,\end{enumerate}
for $x,y,z\in D$, where $\vert x\vert$ is the degree of an element $x\in V$, the degree of the products $\vdash$ and $\dashv$ is $0$, the product $\times $ is of degree $1$ and the differential $d$ is of degree $-1$.\end{definition}
\medskip

Our definition simplifies the signs of F. Chapoton\rq s original one, but it is easy to see that the definitions are equivalent.
\medskip

Let $\G= \G_1\coprod \dots \coprod \G_r$ be a graph, where $\G_i$ is connected for $1\leq i\leq r$. In Definition \ref{tubingdisconnected} we described the tubings of $\G$.
\medskip

\begin{notation}\label{previousdiass} Let $\G= \G_1\coprod \dots \coprod \G_r$ be a graph, where $\G_i$ is connected for $1\leq i\leq r$, and let $T$ be a tubing of $\G$. Recall that, when $r=1$, we denote by ${\overline T}$ the set of all proper tubes of $T$, that is ${\overline T}=T\setminus \{ t_{\G}\}$. For $r > 1$, let $T^c$ denote the union of $T$ and the set $\{ t_{\G_i}\mid 1\leq i\leq r\}$ of all universal tubes of the graphs $\G_i$, for $1\leq i\leq r$.
\medskip

For $r> 1$, and any tubing $T= T_1\coprod \dots \coprod T_r \in {\mbox {Tub}(\G)}$, there exists at least one integer $1\leq k\leq r$ such that $T_k$ is not a tubing of $\G_k$ (that is $T_k\cup \{t_{\G_k}\} \in {\mbox {Tub}(\G_k)}$. Denote by ${\mbox {def}(T)}$ the number of positive integers $1\leq k\leq r$ such that $T_k$ does not contain the universal tube $t_{\G_k}$.

Let $\Omega = \Omega _1\coprod \dots \coprod \Omega _s$ be another graph, with $\Omega _j$ connected for $1\leq j\leq s$. For $T$ be a collection of compatible tubes in $\G$ and $S$ a collection of compatible tubes in $\Omega$, we denote by $T\coprod S$ the collection $T\bigcup (S+n)$ of compatible tubes in $\G\coprod \Omega$, where $n$ is the number of nodes of $\G$.
\end{notation}
\medskip

\begin{definition} \label{defn:triass} Let $\G = \G_1\coprod \dots \coprod \G_r$ and $\Omega = \Omega _1\coprod \dots \coprod \Omega _s$ be two graphs such that $\G_i$ and $\Omega_j$ are connected, for $1\leq i\leq r$ and $1\leq j\leq s$. For any pair of tubings $T\in {\mbox{Tub}(\G)}$ and $S\in {\mbox {Tub}(\Omega)}$, define in ${\mbox {Tub}(\G\coprod \Omega)}$ the tubings:\begin{enumerate}[(a)]
\item $T\vdash S:=\begin{cases} T^c\coprod {\overline S}&, {\rm for}\ s=1\ {\rm and}\ {\mbox {def}(T)}\leq 1,\\
T^c\coprod S&, {\rm for}\ s>1\ {\rm and}\ {\mbox {def}(T)}\leq 1,\\
0 & \ {\rm otherwise},\end{cases}$
\item $T\dashv S:=\begin{cases} {\overline T}\coprod S^c&, {\rm for}\ r=1\ {\rm and}\ {\mbox {def}(S)}\leq 1,\\
T\coprod S^c&, {\rm for}\ r>1\  {\rm and}\ {\mbox {def}(S)}\leq 1,\\
0 & \ {\rm otherwise},\end{cases}$
\item $T\times S := \begin{cases}({\overline T}\coprod {\overline S})&, {\rm for}\ r=s=1,\\
{\overline T}\coprod S&, {\rm for}\ r=1\ {\rm and}\  s>1,\\
T\coprod {\overline S}&, {\rm for}\ r>1\ {\rm and}\  s=1,\\
T\coprod S&, {\rm for}\ r>1\ {\rm and}\  s>1.\end{cases}$\end{enumerate}

The products are extended to all the vector space ${\mbox {\bf DTub}}$ by linearity.\end{definition}
\medskip

The proof of the following result is immediate.

\begin{lemma} \label{diasstruct} The vector space ${\mbox {\bf DTub}}$, equipped with the products $\vdash$, $\dashv$ and $\times$, satisfies conditions $i)$--$vii)$ of Definition \ref{triass}. Moreover, the vector space ${\mbox {\bf DTub}}$ is freely generated on ${\mbox {\bf Tub}}$ by the products $\vdash$, $\dashv$ and $\times$, under the relations $i)$--$vii)$. \end{lemma}
\medskip

As ${\mbox {\bf DTub}}$ is freely generated by the vector space ${\mbox {\bf Tub}}$, and ${\mbox {\bf Tub}}$ is a differential graded vector space, with the differential map $\partial$  described in Definition \ref{defn:boundarymap}, there exists a unique way to extend $\partial $ to a differential map $d$ on ${\mbox {\bf DTub}}$ satisfying conditions $viii)$--$x)$ of Definition \ref{triass}.
\medskip

In Theorem \ref{th:generators}, we proved that ${\mbox {\bf Tub}}$ is a free object spanned by the set ${\mbox {CGraph}}$ of all graded simple connected finite graphs, for an algebraic structure whose operations are given by binary products $\circ _{(\G, t)}$, for any $\G\in {\mbox {CGraph}}$ and any tube $t$ in $\G$, under the relations given in Proposition \ref{propscirc}. For ${\mbox {\bf DTub}}$, we have the following result.
\medskip

\begin{theorem} \label{th:DTub} The vector space ${\mbox {\bf DTub}}$, equipped with the products $\vdash$, $\dashv$ and $\times$ and the differential map $d$, is the free ${\mbox {Trias}\rq}$ algebra generated by the differential graded space $({\mbox {\bf Tub}}, \partial)$ .\end{theorem}
\medskip

\begin{proof} The operad ${\mbox {Trias}\rq}$ is non-symmetric, and the free object spanned by one element is the space ${\mbox {Trias}\rq(\bullet)}$ spanned by the faces of standard simplexes, as shown in \cite{Cha}. So, the underlying vector space of the free ${\mbox {Trias}\rq}$ algebra spanned by a differential graded vector space $(V, \partial)$ is the vector space $\bigoplus_{n\geq 1} {\mbox {Trias}\rq(\bullet)}_n\otimes V^{\otimes n}$.
\medskip

For $n\geq 1$, let $C_n$ be the graph with $n$ nodes and no edge. For $n=1$, let $\circledcirc$ denotes the unique tubing of $C_1$, whose tube is the universal one, and let $\circ$ denotes the empty tubing . The elements of ${\mbox{Tub}(C_n)}$ are identified with the set ${\mbox{}{\mathbb 2}^n-1}$ of proper subsets of $[n]$.

The faces of the $n$ standard simplex ${\mathbb{\Delta}}_n$, of dimension $n{\mbox {-}1}$, are identified with the tubings on $C_n$ in the following way: \begin{enumerate}
\item the face of dimension $n{\mbox {-}1}$ of ${\mathbb{\Delta}}_n$ is the empty tubing on $C_n$, that is the tubing $\circ\coprod \dots \coprod \circ$,
\item the face of dimension $j{\mbox {-}1}$ given by
\begin{align*} \{ (t_1,\dots ,t_n)\in [0; 1]^n\mid t_{i_1}+\dots +t_{i_p} =1,\ {\rm for}&\ 1\leq i_1<\dots <i_p<n,\ {\rm and}\\
 &t_k=0,\ {\rm for}\ k\notin \{i_1,\dots ,i_p\}\},\end{align*}
 is identified with the tubing $T_1\coprod \dots \coprod T_n$ such that $T_{i_l} = \circ$, for $1\leq l\leq p$, and $T_k=\circledcirc$, for $k\notin \{i_1,\dots ,i_p\}$. Denote ${\mathcal T}_{i_1\dots i_j}^n$ the tubing of $C_n$ identified with this face.\end{enumerate}

The set $\{ {\mathcal T}_{i_1\dots i_j}^n\mid \ 1\leq p\leq n\ {\rm and}\ 0\leq i_1<\dots <i_p\}$ is a basis of ${\mbox{Tub}(C_n)}$
\medskip

So, the underlying vector space of ${\mbox {Trias}\rq(\bullet)}$ is spanned, in degree $n$, by the set ${\mbox {Tub}(C_n)}$ of tubings of $C_n$, for $n\geq 1$. It is immediate to see that the products $\vdash $, $\dashv$ and $\times$ are given by the restrictions, to the graphs $C_n\rq $s, of the formulas described in Definition \ref{defn:triass}. The differential $d$ is uniquely defined by the conditions of Definition \ref{triass}.
\medskip

On the other hand, the vector space ${\mbox {\bf Tub}}_m$ has a canonical basis given by $\{ (\G, T)\mid \G \ {\rm is\ a\ connected\ simple\ graph\ with}\ m\ {\rm nodes,\ and}\ T\in {\mbox {Tub}(\G)}\}$.

For a family of elements $(\G_1,T_1),\dots ,(\G_n,T_n)$, let ${\mathcal T}_{i_1\dots i_p}^n(T_1,\dots,T_n)$ be the tubing of $\G_1\coprod\dots \coprod G_n$ defined by:
\begin{equation*}{\mathcal T}_{i_1\dots i_p}^n(T_1,\dots,T_n):=R_1\coprod \dots \coprod R_n,\end{equation*}
where $R_k=\begin{cases} T_k,& {\rm for}\ k\notin\{i_1,\dots ,i_p\},\\
{\overline {T_k}},& {\rm for}\ k\in\{i_1,\dots ,i_p\}.\end{cases}$

By \cite{Cha}, we know that the set
\begin{equation*} {\bold T} := \{{\mathcal T}_{i_1\dots i_p}^n(T_1,\dots,T_n)\mid 1\leq p\leq n,\ 1\leq i_1<\dots <i_p\leq n\},\end{equation*}
for all simple connected graphs  $\G_i$ and all $T_i\in {\mbox {Tub}(\G_i)}$, with $1\leq i\leq p$, gives a basis of the free ${\mbox {Trias}\rq}$ algebra spanned by ${\mbox {\bf Tub}}$, where the degree of ${\mathcal T}_{i_1\dots i_j}^m(T_1,\dots,T_m)$ is the sum of the degrees of the $T_i\rq $s. But clearly, the set ${\bold T}$ is also a basis of ${\mbox {\bf DTub}}$.

Moreover, the products $\vdash$, $\dashv$ and $\times$ on the free algebra ${\mbox {Trias}\rq({\mbox {\bf Tub}})}$ are induced by the ones on ${\mbox {Trias}\rq(\bullet)}$, so they coincide with the products on ${\mbox {\bf DTub}}$, described in Definition \ref{triass}. Finally, as the differential map on ${\mbox {Trias}\rq({\mbox {\bf Tub}})}$ is completely determined by the products, it also coincides with the differential of ${\mbox {\bf DTub}}$, which ends the proof.\end{proof}

\bigskip

\subsection{The $L$- algebra structure of ${\mbox {\bf Tub}}$}
\medskip

There exist many ways to define products on the vector space ${\mbox {\bf Tub}}$, induced by the associative trialgebra structure of ${\mbox {\bf DTub}}$. However, they do not give an associative trialgebra, but an $L$-algebra in the sense of P. Leroux (see \cite{Ler}).

\begin{definition} \label{LerouxL} An {\it L-algebra} is a vector space $L$, equipped with two binary operations $\triangleright$ and $\triangleleft$ satisfying that:
\begin{equation*}x\triangleright (y\triangleleft z) = (x\triangleright y)\triangleleft z,\end{equation*}
for any elements $x, y, z\in L$.\end{definition}

Let $\G$ and $\Omega$ be two connected graphs, with $\vert{\mbox{Nod}(\G)}\vert=n$ and $\vert{\mbox{Nod}(\Omega)}\vert=m$. For any pair of tubings $T\in {\mbox {Tub}(\G)}$ and $S\in {\mbox {Tub}(\Omega)}$, consider the sets $X_T:=\{1\leq i_1<\dots <i_k\leq n\}$ of nodes of $\G$ which do not belong to any proper tube of $T$ and $X_S+n:= \{ n+1\leq j_1<\dots <j_h\leq m\}$ of nodes of $\Omega +n$ which do not belong to any proper tube of $S$. The connected graph $\G\cup_{T,S}\Omega$ is the graph obtained by adding to the graph $\G\coprod \Omega$ the edge $(i_k,j_1)$.
\medskip

The operations $\vdash$, $\dashv$ and $\times$ of ${\mbox {\bf DTub}}$ induce operations on ${\mbox {\bf Tub}}$, described by:\begin{enumerate}[(a)]
\item $T\triangleright S:= \{ t\in T\} \bigcup \{s+n\mid s\in {\overline S}\}\bigcup \{t_{\G\cup_{T,S}\Omega}\}$,
\item $T\triangleleft S:= \{t\mid t\in {\overline T}\}\bigcup \{s+n\in S+n\}\bigcup \{t_{\G\cup_{T,S}\Omega}\}$,
\item $T\perp S:=\{t\mid t\in {\overline T}\}\bigcup \{s+n\mid s\in {\overline S}\}\bigcup \{t_{\G\cup_{T,S}\Omega}\}$,\end{enumerate}
in ${\mbox {Tub}(\G\cup_{T,S}\Omega)}$.
\medskip

Clearly, the products $\triangleright$ and $\triangleleft$ defined on ${\mbox {\bf Tub}}$ are not associative, and do not satisfy the first two conditions of Definition \ref{triass}. However, we get the following result, whose proof follows from a straightforward calculation.

\begin{lemma}\label{Lalgrs} The vector space ${\mbox {\bf Tub}}$, equipped with the binary operations $\triangleright$ and $\triangleleft$ is an L-algebra. The product $\perp$ on ${\mbox {\bf Tub}}$ is associative, and it satisfies the following relations:\begin{enumerate}[(i)]
\item $T\triangleright(S\perp W) = (T\triangleright S)\perp W$,
\item $T\perp (S\triangleright W) = (T\triangleleft S)\perp W$,
\item $T\perp (S\triangleleft W) = (T\perp S)\triangleleft W$,\end{enumerate}
for elements $T,S$ and $W$ in ${\mbox {\bf Tub}}$.\end{lemma}
\bigskip

\section{An operadic category associated to graph associahedra}
\medskip

Let ${\mbox {gVect}_{\K}}$ denotes the category of graded vector spaces over $\K$.

Given a plane rooted tree $Y$ and a graded vector space $V$, the vector space $Y(V)$ is
\begin{equation*} Y(V):=\bigotimes_{v\in {\mbox{Vert}(Y)}}V_{\vert v\vert},\end{equation*}
where ${\mbox{Vert}(Y)}$ denotes the set of internal vertices of the tree $Y$ and $\vert v\vert$ is the number of inputs of the vertex $v$.

The previous construction defines an functor ${\mathcal Y}: {\mbox {gVect}_{\K}}\longrightarrow {\mbox {gVect}_{\K}}$, given by:
\begin{equation*} {\mathcal Y}(V)=\bigoplus _{n\geq 0} \bigoplus_{\vert {\mbox{Lea}(Y)}\vert=n}Y(V),\end{equation*}
where ${\mbox{Lea}(Y)}$ denotes the number of leaves of a plane rooted tree $Y$.

The functor ${\mathcal Y}$ is a monad in the category of endofunctors of ${\mbox {gVect}_{\K}}$, equipped with the composition (see \cite{Ma}, \cite{GiKa}, \cite{Le}, \cite{MSS}, \cite{LV}) which defines non-symmetric operads.

In \cite{LR}, J.-L. Loday and the second author constructed another monad in the category of endofunctors of ${\mbox {gVect}_{\K}}$, where plane rooted trees were replaced by surjective maps, that is the set $\bigcup_{n\geq 1}{\mbox{Tub}(K_n)}$ of tubings of all complete graphs.
\medskip

Following the previous construction, we also may define a functor ${\mathbb P}:{\mbox {gVect}_{\K}}\longrightarrow {\mbox {gVect}_{\K}}$ as follows:\begin{enumerate}
\item for any tube $t$ is a tubing $T$ in $\G$, define the {\it arity} of $t$ as the number of nodes in $t$ which do not belong to any other tube $t\rq \in T\vert_t$ plus one. We denote the arity of $t$ by ${\mbox{ar}(t)}$.
\item for any tubing $T$ of $\G$ and any graded vector space $V$, let
$V_T:= \bigotimes _{t\in T}V_{\mbox{ar}(t)}.$
\item define
\begin{equation*} {\mathbb P}(V)_{n+1} := \bigoplus _{\vert {\mbox {Nod}(\G)}\vert = n}\left (\bigoplus _{T\in {\mbox {Tub}(\G)}} V_T\right ).\end{equation*}\end{enumerate}
However,  we do not get a monad structure ${\mathbb P}\circ {\mathbb P}\longrightarrow {\mathbb P}$, due to:\begin{enumerate}
\item ${\mathbb P}$ is not unital, because there exist many finite simple connected graphs with $n$ vertices, for any fixed $n\geq 3$,
\item the composition ${\mathbb P}\circ {\mathbb P}$ is not always defined. For a graph $\G$ and a tubing $T\in {\mbox{Tub}(\G)}$, we get that
\begin{equation*} {\mathbb P}({\mathbb P}(V))_T=\bigotimes_{t\in T}\bigl( \bigoplus _{\vert {\mbox {Nod}(\Omega_t)}\vert = {\mbox{ar}(t)}-1}\left (\bigoplus _{S_t\in {\mbox {Tub}(\Omega_t)}} V_{S_t}\right )\bigr).\end{equation*}
But, in order to apply substitution and get an element in ${\mathbb P}(V)$, we need that $\Omega_t=(\G_t)_{\mbox{Maxt}(T\vert_t)}^*$, for any $t\in T$.\end{enumerate}
\medskip

In \cite{BaMa}, M. Batanin and M. Markl introduced the notion of operadic category, which allows them to provide a large generalization of operads (see \cite{MSS}, \cite{Ma}, \cite{Le} and \cite{LV}). In \cite{Mark}, M. Markl defined an operadic category whose operads are precisely permutads.

The aim of the present section is to introduce a category, whose objects are tubings on finite graphs, and to show that substitution on graph associahedra provides an example of Batanin and Markl\rq s strict operadic category.
\medskip

\subsection{Operadic categories}
\medskip

We follow the definition of operadic category given by M. Batanin and M. Markl in \cite{BaMa}, for a more detailed description, examples and applications we refer to their publication. In all the section, ${\mathcal C}$ denotes a complete and cocomplete closed symmetric monoidal category, whose monoidal product is denoted by $\otimes$ and whose unit is denoted $I$.
\medskip

\begin{definition} \label{sFSet} The operad ${\mbox {sFSet}}$ is the operads whose objects are the linearly ordered sets $[n]=\{1<2<\dots <n\}$, for $n\geq 1$, and whose homomorphism are map between finite sets, which do not necessarily preserve the order.\end{definition}

The terminal object of ${\mbox {sFSet}}$ is $[1]$.

\begin{definition}\label{fiber} Given two maps $f\in {\mbox {sFSet}([n], [m])}$ and $i\in {\mbox {sFSet}([1], [m])}$ the {\it $i^{\text{th}}$ fiber} of $f$ on $i$ is the pull back $f^{-1}(i)$ of the diagram:
\begin{equation*} \begin{matrix}
f^{-1}(i)&\longrightarrow&[n]\\
\downarrow &\qquad &\downarrow\\
[1]&\longrightarrow &[m]\end{matrix}
\end{equation*}
where $f^{-1}(i)$ is identified with the finite set $\{j\mid 1\leq j\leq n,\ f(j)=i\}$ renumbered using the order induced by the linear order of $[f^{-1}(i)]$.\end{definition}

Strict operadic categories are categories with a functor to ${\mbox {sFSet}}$, having a family of terminal objects (one for each connected component of the category) and certain inverse images, which behave like pull-backs.

\begin{definition} \label{strictoperadiccat} A {\it strict operadic category} is a category ${\mathcal O}$ together with:\begin{enumerate}[(a)]
\item a fixed family of terminal objects $U_c$, for each connected component $c\in \pi_0( {\mathcal O})$,
\item a {\it cardinality} functor $\vert -\vert : {\mathcal O}\longrightarrow {\mbox {sFSet}}$,
\item an object $f^{-1}(i)$ in ${\mathcal O}$ such that $\vert f^{-1}(i)\vert =\vert f\vert^{-1}(i)$, for every pair of homomorphisms $f\in {\mathcal O}(T,S)$ and every element $i\in \vert S\vert$,\end{enumerate}
satisfying that \begin{enumerate}[(i)]
\item $\pi_0({\mathcal O})$ is small.
\item $\vert U_c\vert =[1]$, for any $c\in \pi_0({\mathcal O})$.
\item for any object $T$ in ${\mathcal O}$, the identity $1_T\in {\mathcal O}(T,T)$ satisfies that $1_T^{-1}(i)= U_{c_i}$, for all $i\in \vert T\vert$ and some $c_i\in \pi_0({\mathcal O})$,
\item for any pair of morphisms $f\in {\mathcal O}(T,S)$ and $g\in {\mathcal O}(S,R)$, and every $i\in \vert R\vert$, there exists $f_i: (g\circ f)^{-1}(i)\longrightarrow g^{-1}(i)$, such that
$\vert f_i\vert$ is the natural map $\vert g\circ f\vert ^{-1}(i)\longrightarrow \vert g\vert ^{-1}(i)$ in ${\mbox {sFSet}}$. Moreover, the assignment ${\mbox {Fib}_i}(g) :=g^{-1}(i)$ gives a functor ${\mbox {Fib}_i}: {\mathcal O}/R\longrightarrow {\mathcal O}$, which is the domain functor for $R=U_c$.
\item for any pair of morphisms $f\in {\mathcal O}(T,S)$ and $g\in {\mathcal O}(S,R)$, and any $j\in \vert S\vert$, we have that $f^{-1}(j) = f_i^{-1}(j)$, where $i= \vert g\vert (j)$.
\item given three homomorphisms $f\in {\mathcal O}(T,S)$, $g\in {\mathcal O}(S,Q)$ and $h\in {\mathcal O}(Q,R)$, and $i\in \vert R\vert$, the condition above states that there exists morphisms\begin{enumerate}
\item $(g\circ f)_i = g_i\circ f_i: (h\circ g\circ f)^{-1}(i)\longrightarrow h^{-1}(i),$
\item $f_i:(h\circ g\circ f)^{-1}(i)\longrightarrow (h\circ g)^{-1}(i)$,
\item $g_i:(h\circ g)^{-1}(i)\longrightarrow h^{-1}(i).$\end{enumerate}

For $j\in \vert Q\vert$, such that $\vert h\vert (j)=i$, we have that $g^{-1}(j) =(g_i)^{-1}(j)$ and $(g\circ f)^{-1}(j) = (g\circ f)_i^{-1}(j).$

As $\vert h^{-1}(i)\vert =\vert h\vert ^{-1}(i)$ is the inverse image (reordered) in ${\mbox {sFSet}}$,we get that $j\in \vert h^{-1}(i)\vert =\vert h\vert ^{-1}(i)$.

So, there exists $(f_i)_j:(g\circ f)_i^{-1}(j)\longrightarrow g_i^{-1}(j)$, the equality $f_j=(f_i)_j$ is required.\end{enumerate}\end{definition}

\begin{definition} \label{Ocollection} Given an strict operadic category ${\mathcal O}$ and a monoidal category ${\mathcal C}$, an {\it ${\mathcal O}$-collection} in ${\mathcal C}$ is a collection $\{E(T)\}_{T\in {\mathcal O}}$ of objects of ${\mathcal C}$, indexed by the objects of the category ${\mathcal O}$. For an ${\mathcal O}$-collection $E$ in ${\mathcal C}$ and a homomorphism $f:T\longrightarrow S$ in ${\mathcal O}$, the object $E(f)$ in ${\mathcal C}$ is defined by
\begin{equation*}E(f):=\bigotimes_{i\in \vert S\vert}E(f^{-1}(i)).\end{equation*}
\end{definition}

Operadic categories provide a good framework to define operadic composition.

\begin{definition} \label{operads} An {\it ${\mathcal O}$-operad} in ${\mathcal C}$ is an ${\mathcal O}$-collection ${\mathbb P}=\{{\mathbb P}(T)\}_{T\in {\mathcal O}}$ in ${\mathcal C}$, equipped with:\begin{enumerate}[(a)]
\item units $I\longrightarrow {\mathbb P}(U_c)$, for $c\in \Pi_0({\mathcal C})$
\item structure maps $\mu (f): {\mathbb P}(S)\otimes {\mathbb P}(f)\longrightarrow {\mathbb P}(T)$, for any $f\in {\mathcal O}(T,S)$,\end{enumerate}
satisfying \begin{enumerate}[(i)]
\item For any pair of homomorphisms $f\in {\mathcal O}(T,S)$ and $g\in {\mathcal O}(S,R)$, the following diagram commutes:
\begin{equation*}
\begin{matrix} \bigotimes_{i\in \vert R\vert}{\mathbb P}(R)\otimes {\mathbb P}(g)\otimes {\mathbb P}(f_i)&{\overset{{\small{\mbox{id}}\otimes \otimes_i\mu(f_i)}}{\longrightarrow}}&{\mathbb P}(R)\otimes {\mathbb P}(h)\\
{ {\mu(g)\otimes {\mbox{id}}}{\downarrow}} &\qquad\qquad &{\downarrow}{\small{\mu(h)}}\\
{\mathbb P}(R)\otimes {\mathbb P}(f)&{\overset{\small{\mu (f)}}{\longrightarrow}}&{\mathbb P}(T),\end{matrix}\end{equation*} where the $f_i$\rq s are the morphisms introduced in Definition \ref{strictoperadiccat}.
\item The compositions
\begin{equation*} {\mathbb P}(T)\longrightarrow\bigotimes_{i\in \vert T\vert }I\otimes {\mathbb P}(T)\longrightarrow\bigotimes_{i\in \vert T\vert} {\mathbb P}(U_{c_i})\otimes {\mathbb P}(T)\overset{\mu({\mbox{id}})}{\longrightarrow} {\mathbb P}(T),\end{equation*}
and
\begin{equation*} {\mathbb P}(T)\otimes I\longrightarrow{\mathbb P}(T)\otimes {\mathbb P}(U_c)\overset{\mu(p_c)}{\longrightarrow }{\mathbb P}(T),\end{equation*}
where $c:=\pi_0(T)$ and $p_c$ is the unique element of ${\mathcal O}(T,U_c)$, are the identity.\end{enumerate}\end{definition}
\medskip

\subsection{The category ${\mathcal O}_{CD}$}
\medskip

We want to describe the operadic category ${\mathcal O}_{CD}$ such that the substitution of tubings defined in Section 2 provides a natural example of ${\mathcal O}_{CD}$ operad. Our model is M. Markl\rq s operadic category ${\mbox {Per}}$, described in \cite{Mark}.
\medskip

\begin{definition} \label{operadiccatCD} Define the category ${\mathcal O}_{CD}$ as follows:\begin{enumerate}
\item The objects of ${\mathcal O}_{CD}$ are pairs $(\G, T)$, where $\G$ is a connected simple finite graph and $T$ is a tubing of $\G$.
\item  The homomorphisms in ${\mathcal O}_{CD}$ are given by:
\begin{equation*}{\mathcal O}_{CD}((\G,T), (\Omega , S)):=\begin{cases} \emptyset ,&{\rm for}\ \G\neq \Omega\ {\rm or}\ T\not\preceq S,\\
\{\iota _{\G, T,S}\},&{\rm for}\ \G=\Omega\ {\rm and}\ T\preceq S,\end{cases}\end{equation*} where $T\preceq S$ is described in Definition \ref{ordertubings}, and means that $T$ is obtained from $S$ by adding compatible tubes.\end{enumerate}
\end{definition}
\medskip

Note that $\pi_0({\mathcal O}_{CD})= {\mbox {CGraph}}$, the set of all simple connected finite graphs equipped with a total order on the set of nodes, and the terminal objects of ${\mathcal O}_{CD}$ are $(\G, T_{\G})$, for $\G\in {\mbox {CGraph}}$.
\medskip

Let $\G$ be a finite connected simple graph and let $T$ be a tubing of $\G$, we denote by ${\mathbb L}(T)$ the number of tubes of $T$. The definition of a functor from the category of tubings to the category ${\mbox {sFSet}}$, requires a standard way to enumerate tubes in a tubing $T$.

\begin{definition} \label{frakN} Let $T$ be a tubing  of a connected simple graph $\G$. The map ${\mathfrak N}_T: \{ t\in T\}\longrightarrow [{\mathbb L}(T)]$ is defined recursively, as follows:\begin{enumerate}
\item If $T=T_{\G}$, then ${\mathfrak N}_T(t_{\G})= 1$.
\item Suppose that $T$ has more than one tube. Let $X= \{ t_1,\dots ,t_{r}\}$ be the set of tubes in $T$ which do not contain another tube of $T$, ordered in such a way that the minimal node of $t_i$ is smaller than the minimal node of $t_{i+1}$, for $1\leq i<r$. In this case ${\mathfrak N}_T(t_i) = i$, for $1\leq i\leq r$.
\item Consider the reconnected complement $\G_{t_1,\dots ,t_{r}}^*:= (((\G_{t_1}^*)_{t_2}^*)\dots )_{t_{r}}^*$ and let $T_1$ be the tubing induced by $T$ on $\G_{t_1,\dots ,t_{r}}^*$. As
${\mathbb L}(T_1)<{\mathbb L}(T)$, we may suppose that ${\mathfrak N}_{T_1}$ is defined. We define
\begin{equation*} {\mathfrak N}_T(t) = {\mathfrak N}_{T_1}(t)+r,\end{equation*}
for $t\notin X_1$.\end{enumerate}\end{definition}

The map which sends the object $(\G,T)$ to the set $\vert (\G,T)\vert := [{\mathbb L}(T)]$ induces a functor ${\mathcal O}_{CD}\longrightarrow {\mbox {sFSet}}$. We get that
\begin{equation*} \vert \iota_{\G, T,S}\vert (k) = {\mathfrak N}_S(t\rq),\end{equation*}
where $k={\mathfrak N}_T(t)$ and $t\rq$ is the minimal tube in $S$ which contains $t$.
\medskip

For $f= \iota_{\G, T,S}$ and $i\in \vert (\G, S)\vert $, there exists a unique tube $s\in S$ such that ${\mathfrak N}_S(s)=i$. Define $f^{-1}(i)$ as:
\begin{equation*} f^{-1}(i)=\left((\G_s)_{(S\vert_s)}^*, (T\vert_s)_{(S\vert_s)}^*\right),\end{equation*}
where $(\G_s)_{(S\vert_s)}^*$ is the reconnected complement of $\G_s$ by the set ${\mbox{Maxt}(S\vert_s)}$ of proper maximal tubes of $S\vert_s$, as described in Definition \ref{maxtubings}, and $(T\vert_s)_{(S\vert_s)}^*$ denotes the tubing induced by $T\vert_s$ on $(\G_s)_{(S\vert_s)}^*$.
\medskip

For example, consider

\begin{figure}[h]

\includegraphics[width=4.5in]{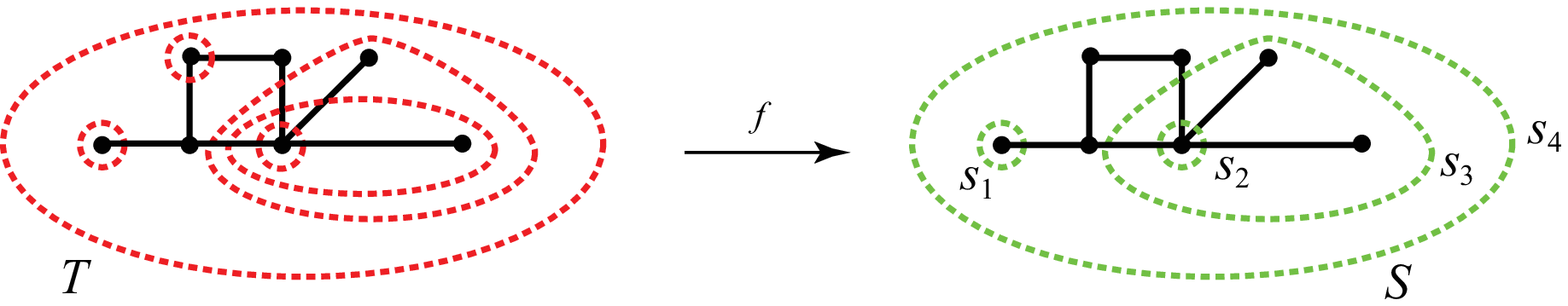}
\end{figure}

we get that

\begin{figure}[h]

\includegraphics[width=4.5in]{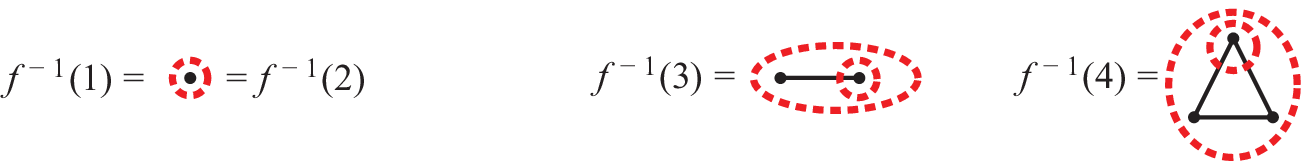}
\end{figure}

\begin{proposition}\label{propstrictopcat} The category ${\mathcal O}_{CD}$ is a strict operadic category .\end{proposition}

\begin{proof} We know that $\pi_0({\mathcal O}_{CD})$ is the set of all simply connected finite graphs whose set of nodes is totally ordered. For any simple connected graph $\G\in {\mathcal O}_{CD}$, the terminal object $U_{\G}$ is the pair $(\G, T_{\G})$. Clearly, we have that $\vert (\G, T_{\G})\vert =1$.
\medskip

Suppose that $(\G, T)$ is an object of ${\mathcal O}_{CD}$, where $T$ has $k$ tubes. For any $1\leq i\leq k$, we have that the identity map $1_{(\G,T_{\G})}^{-1}(i)$ is
\begin{equation*}((\G_{t_i})_{(T\vert_{t_i})}^*, (T\vert_{t_i})_{(T\vert_{t_i})}^*)=(\G_{t_i})_{(T\vert_{t_i})}^*, T_{(\G_{t_i})_{(T\vert_{t_i})}^*})= U_{\G_{t_i})_{(T\vert_{t_i})}^*},\end{equation*}
where $t_i$ is the unique tube in $T$ such that ${\mathfrak N}_T(t_i) = i$.
\medskip

Let now $\G$ be a graph, and $T\preceq S\preceq R$ be three tubings of $\G$. For an element $1\leq i\leq \vert R\vert$, there exists a unique tube $r_i$ in $R$, satisfying that ${\mathfrak N}_R(r_i) = i$. As $T\preceq S\preceq R$, we get that $r_i$ belongs to $S$ and to $T$, and that ${\mathfrak N}_R(r_i) \leq {\mathfrak N}_S(r_i)\leq {\mathfrak N}_T(r_i)$.

We have that ${\mathcal O}_{CD}((\G, T),(\G,S))=\{\iota_{\G, T, S}\}$ and ${\mathcal O}_{CD}((\G, S),(\G,R))=\{\iota_{\G, S, R}\}$. It suffices to prove condition $(iv)$ of
Definition \ref{strictoperadiccat} for $f=\iota_{\G, T, S}$ and $g=\iota_{\G, S, R}$. In this case, we get that $g\circ f =\iota_{\G, T, R}$.

So, for any $1\leq i\leq \vert R\vert$, we get that $g^{-1}(i) = ((\G_{r_i})_{R\vert_{r_i}}^*, (S\vert_{r_i})_{(R\vert_{r_i})}^*)$ and $(g\circ f)^{-1}(i)= ((\G_{r_i})_{R\vert_{r_i}}^*, (T\vert_{r_i})_{(R\vert_{r_i})}^*)$, where ${\mathfrak N}_R(r_i)=i$.

As $T\preceq S$, we have that $(T\vert_{r_i})_{(R\vert_{r_i})}^*\preceq (S\vert_{r_i})_{(R\vert_{r_i})}^*$, so
\begin{equation*} f_i=\iota_{(\G_{r_i})_{R\vert_{r_i}}^*, (T\vert_{r_i})_{(R\vert_{r_i})}^*, (S\vert_{r_i})_{(R\vert_{r_i})}^*}.\end{equation*}

When ${\mathcal O}_{CD}((\G, T),(\G,S))=\emptyset $ or ${\mathcal O}_{CD}((\G, S),(\G,R))=\emptyset$, the result is immediate because there do not exist two morphisms $f$ and $g$.

Clearly, we get that the assignment ${\mbox {Fib}_i(g)}:=g^{-1}(i)$ gives a functor ${\mbox {Fib}_i}:{\mathcal O}_{CD}/(\G,R)\longrightarrow {\mathcal O}_{CD}$, which is the domain functor for $(\G,R) = (\G, T_{\G})$, for any graph $\G$.
\medskip

Let $j\in \vert S\vert$ be such that $\vert g\vert (j)=i$. There exists a unique tube $s_j\in S$ satisfying that ${\mathfrak N}_S(s_j)=j$ and that the minimal tube of $R$ which containing $s_j$ is $r_i$.

We have that $f^{-1}(j)=((\G_{s_j})_{S\vert _{s_j}}^*, (T\vert_{s_j})_{S\vert _{s_j}}^*)$.

On the other hand, $f_i: ((\G_{r_i})_{R\vert_{r_i}}^*, (T\vert_{r_i})_{R\vert_{r_i}}^*) \longrightarrow ((\G_{r_i})_{R\vert_{r_i}}^*, (S\vert_{r_i})_{R\vert_{r_i}}^*)$, is the map
$\iota_{(\G_{r_i})_{R\vert_{r_i}}^*, (T\vert_{r_i})_{R\vert_{r_i}}^*,  (S\vert_{r_i})_{R\vert_{r_i}}^*}$.

Note that, as $S\preceq R$, $s_j\subseteq r_i$ and any tube $r$ of $R$ containing $s_j$ satisfies that $r_i\subseteq r$, we have that $((\G_{r_i})_{R\vert_{r_i}}^*)_{s_j})_{S\vert_{s_j}}^* =\G_{s_j})_{S\vert _{s_j}}^*$.

Moreover, the restriction of the tubing $(T\vert_{r_i})_{R\vert_{r_i}}^*$ to $(\G_{s_j})_{S\vert _{s_j}}^*$ coincides with the tubing $(T\vert_{s_j})_{S\vert_{s_j}}^*$. Therefore, $f_i^{-1}(j)= ((\G_{s_j})_{S\vert _{s_j}}^*, (T\vert_{s_j})_{S\vert _{s_j}}^*)= f^{-1}(j)$.
\medskip

To prove the last condition of Definition \ref{strictoperadiccat}, consider a graph $\G$ and tubings $T\preceq S\preceq Q\preceq R$ of $\G$. We get that $f=\iota_{\G, T, S}$, $g=\iota_{\G, S, Q}$  and $h=\iota_{\G, Q, R}$. If we do not have that $T\preceq S\preceq Q\preceq R$, then the triple of morphisms $(f,g,h)$ does not exist.
\medskip

Let $1\leq j\leq \vert Q\vert$. There exists a unique tube $q_j$ in $Q$ such that ${\mathfrak N}_Q(q_j)=j$. If $\vert \iota_{\G,Q, R}\vert (j)=i$, for some $1\leq i\leq \vert R\vert$, then there exists a unique tube $r_i\in R$ such that $r_i$ is the minimal tube in $R$ which contains $q_j$.

We have that
\begin{equation*} f_i = \iota_{(\G_{r_i})_{R\vert_{r_i}}^*, (T\vert_{r_i})_{(R\vert_{r_i})}^*, (S\vert_{r_i})_{(R\vert_{r_i})}^*}: ((\G_{r_i})_{R\vert_{r_i}}^*, (T\vert_{r_i})_{(R\vert_{r_i})}^*)\longrightarrow ((\G_{r_i})_{R\vert_{r_i}}^*, (S\vert_{r_i})_{(R\vert_{r_i})}^*),\end{equation*}
 and $h^{-1}(i)= ((\G_{r_i})_{R\vert_{r_i}}^*, (Q\vert_{r_i})_{(R\vert_{r_i})}^*)$.

As $j\in \vert h\vert ^{-1}(i)$, the morphism $(f_i)_j$ is the restriction of $f_i$ to $(g\circ f)^{-1}(j) = ((\G_{q_j})_{Q\vert_{q_j}}^*, (T\vert_{q_j})_{(Q\vert_{q_j}}^*)$. So, $(f_i)_j=\iota_{(\G_{q_j})_{Q\vert_{q_j}}^*, (T\vert_{q_j})_{Q\vert_{q_j}}^*,  (S\vert_{q_j})_{Q\vert_{q_j}}^*}$, which coincides with the definition of $f_j$. \end{proof}

\begin{remark} \label{rem:catPer} Consider the full subcategory ${\mathcal O}_K$ of ${\mathcal O}_{CD}$ whose objects are pairs $(K_n, T)$, where $K_n$ is the complete graphs with $n$ nodes. As shown in Example \ref{excomplete graph}, $T$ is given by a family of tubes $t^1\subsetneq t^2\subsetneq \dots \subsetneq t^{r-1}\subsetneq [n]$.

Therefore, the tubing $T$ is identified with the map $f_T:[n]\mapsto [r]$, given by $f_T(i)= j$, where $j$ is the minimal integer satisfying that the node $i$ belongs to $t_j$ (that is, $i\notin t_{j-1}$ and $i\in t_j$).
\medskip

Suppose that we have a morphisms $\varphi: (K_n, T)\longrightarrow (K_n, S)$ in ${\mathcal O}_{CD}$, and that $S=\{s_1\subsetneq \dots s_{k-1}\}\subsetneq [n]$.
As $S\subseteq T$, we have that $k\leq r$.

The map $g: [r]\longrightarrow [k]$, defined by $g(i)=j$, where $s_j$ is the minimal tube of $S$ containing $t_i$, for $1\leq i\leq r$, is clearly surjective, and satisfies that $g\circ f_T=f_s$.
\medskip

Therefore, the category ${\mathcal O}_K$ is equivalent to the category ${\mbox{Per}}$, defined in 14.4 of \cite{BaMa1} and in \cite{Mark}. \end{remark}

\end{document}